\documentclass[a4paper,usenames,10pt]{amsart}
\usepackage[utf8]{inputenc}
\usepackage[english]{babel}
\usepackage{mathrsfs}
\usepackage{amsthm}
\usepackage{amsfonts}
\usepackage{amsmath}
\usepackage{amssymb}
\usepackage{faktor}
\usepackage{mathtools}
\usepackage{tikz}
\usepackage{esint}
\usepackage{centernot}
\usepackage{verbatim}
\usepackage[a4paper,top=3.5cm,bottom=3.5cm,left=2.5cm,right=2.5cm]{geometry}

\usepackage{cite}
\usepackage{color}
\definecolor{citegreen}{rgb}{0,0.8,0}
\definecolor{refred}{rgb}{0.8,0,0}
\usepackage[colorlinks, citecolor=citegreen, linkcolor=refred]{hyperref}

\newtheorem{theorem}{Theorem}
\newtheorem{lemma}[theorem]{Lemma}
\newtheorem{corollary}[theorem]{Corollary}
\newtheorem{proposition}[theorem]{Proposition}
\theoremstyle{definition}
\newtheorem{remark}[theorem]{Remark}

\newtheorem{definition}[theorem]{Definition}

\newcommand{\R}{\mathbb R}
\numberwithin{equation}{section} \numberwithin{theorem}{section}
\allowdisplaybreaks
\newcounter{stepctr}
{\end{list}}

\def\XXint#1#2#3{{\setbox0=\hbox{$#1{#2#3}{\int}$}
     \vcenter{\hbox{$#2#3$}}\kern-.5\wd0}}

\DeclareMathOperator{\Vol}{Vol}
\DeclareMathOperator{\Area}{Area}
\DeclareMathOperator{\Length}{Length}
\DeclareMathOperator{\Sc}{R}
\DeclareMathOperator{\Ric}{Ric}
\DeclareMathOperator{\Rm}{Rm}
\DeclareMathOperator{\inj}{inj}
\DeclareMathOperator{\Sh}{S}
\DeclareMathOperator{\Hess}{Hess}
\newcommand{\e}{\varepsilon}

\newcommand{\set}[1]{\{ {#1} \}}

\newcommand{\scal}[2]{\langle {#1} , {#2} \rangle}

\DeclareUnicodeCharacter{2212}{-}

\title{Analysis of Type I Singularities in the Harmonic Ricci Flow}
\author{Gianmichele Di Matteo}

\begin{document}

\begin{abstract}
In \cite{end1}, Enders, M\"uller and Topping showed that any blow up sequence of a Type I Ricci flow near a singular point converges to a non-trivial gradient Ricci soliton, leading them to conclude that for such flows all reasonable definitions of singular points agree with each other. We prove the analogous result for the harmonic Ricci flow, generalizing in particular results of Guo, Huang and Phong \cite{guo} and Shi \cite{shi}. In order to obtain our result, we develop refined compactness theorems, a new pseudolocality theorem, and a notion of reduced length and volume based at the singular time for the harmonic Ricci flow.
\end{abstract}

\maketitle

\section{Introduction}
Geometric flows constitute a general procedure of transforming a given object into an equivalent one with more symmetry or regularity properties. They first appeared in $1964$ in \cite{eel}, where Eels and Sampson introduced the harmonic map flow to solve the homotopy problem in non-positively curved manifolds. Another remarkable example is provided by the solution of the Poincar\'e and Thurston Geometrization Conjectures by Perelman in \cite{per} and \cite{per2} using Ricci flow. In general, geometric flows may develop singularities which may prevent them from being continued in any reasonable way; in these cases, one may hope to understand the geometry at the singular time well enough to be able to modify the object in consideration without losing any information, and then to start the flow again. It is therefore natural to study the singularity formation of these kind of flows.

In \cite{mul}, M\"uller introduced the harmonic Ricci flow, a coupling between the two flows mentioned above, developed for it short and long time existence theory and monotonicity formulas for entropy functionals, reduced distance and volume, and a non-collapsing theorem reminiscent of Perelman's tools for the Ricci flow introduced in \cite{per}. 

Suppose we are given two Riemannian manifolds $(M^n,g)$, $(N^k,\gamma)$, with $N$ closed, and a non-negative \textit{coupling function} $\alpha(t)$. We say that a smooth family $(g(t),\phi(t))$ of Riemannian metrics $g(t)$ on $M$ and smooth maps $\phi(t)$ from $M$ to $N$ provides a solution to the harmonic Ricci flow if it satisfies
\begin{equation}
\begin{cases}
\frac{\partial g(t)}{\partial t}=-2 \Ric_{g(t)}+2\alpha(t) \nabla \phi(t) \otimes \nabla \phi(t), \\
\frac{\partial \phi(t)}{\partial t}=\tau_{g(t)} \phi(t). 
\end{cases}
\label{hrf}
\end{equation}
Here we have denoted by $\tau_g \phi= \Delta_g \phi-A(\phi)g(\nabla \phi, \nabla \phi)$ the so-called \emph{tension field} of $\phi$, where $A$ is the second fundamental form of $N$, viewed isometrically embedded into $\R^d$ via Nash's embedding theorem. In the special case where $\phi(t)$ is a real valued function, this flow reduces to List's flow introduced in \cite{lis}. For convenience of the reader, we introduce the $(0,2)-$tensor $\mathcal{S} \coloneqq \Ric-\alpha \nabla \phi \otimes \nabla \phi$, locally denoted $S_{ij}$, and its trace $\Sh \coloneqq \Sc-\alpha |\nabla \phi|^2$.

A priori, a solution to (\ref{hrf}) can develop a singularity either in the metric component or the map component only or in both components simultaneously. Surprisingly, it turns out that the metric component dominates the map component and certain types of singularities can be ruled out. In particular, for sufficiently large coupling functions $\alpha(t)$, singularities in the map component can be excluded, see Proposition 5.6 in \cite{mul}. (In fact, singularity formation in finite time for a volume-normalised version of (\ref{hrf}) can be excluded completely for large $\alpha(t)$ if the domain manifold $M$ is two dimensional, see \cite{buz}.) More generally, as long as $\alpha(t)$ is bounded away from zero, the map component cannot become singular without the metric component also developing a singularity. That is, the coupled system has a singularity at a finite time $T$ if and only if the curvature blows up at $T$, see Corollary 5.3 in \cite{mul}. These results should be compared with our results in Section \ref{typeItypeA}. 

\noindent In the present paper we are dealing with Type I harmonic Ricci flows, which we define as follows.
\begin{definition}
We say that a solution $(g(t),\phi(t))$ of (\ref{hrf}) on $[0,T)$, $T<\infty$, is a \textit{Type I harmonic Ricci flow} if there exists a constant $C>0$ such that
\begin{equation}
|\Rm|(x,t) \le \frac{C}{(T-t)}.
\end{equation}
\end{definition}
\noindent In \cite{shi}, Shi gave a definition of Type I harmonic Ricci flow in terms of the so-called AC curvature, which depends both on the evolving metric and the map. Our definition is different and depends only on the metric component of the flow, in agreement with the remarks above.

The main reason for introducing this concept resides in the fact that if the flow develops a finite time singularity, this is forming as fast as possible, implying heuristically an easier blow up analysis; these singularities are called Type I singularities.

In Ricci flow theory, Type I Ricci flows are very well studied. The work most relevant for us is \cite{end1}, where Enders, M\"uller and Topping proved self-similarity and non-triviality of the limits of blow-up sequences at singular space-time points for these flows, confirming a result conjectured by Hamilton in \cite{ham1}. The first positive results in this direction were made by Sesum in \cite{ses} in the case where the blow-up limit is compact and, in the general case, blow-up to a gradient shrinking soliton was proved by Naber \cite{nab}.
Moreover, Enders, M\"uller and Topping also showed the equivalence of every natural definition of singular point for Type I Ricci flows. Their methods of proof rely upon Perelman's pseudolocality theorem \cite{per} and an extension of Perelman's reduced length, obtained independently by Enders \cite{end} and Naber \cite{nab}. This extension has the crucial property of being based at the singular time, so one can use it to analyze the blow up limit at a singular point, forcing the limit to be a gradient shrinking soliton in canonical form. The former is then used to prevent the trivial limit to occur in case the blow up is done at a singular point. The other results stated are also consequences of the pseudolocality theorem. It is worth mentioning another approach to this classification, based on the entropy functional, is developed in \cite{man} by Mantegazza and M\"uller.

The aim of our work is to prove an analogue of these results for the harmonic Ricci flow. Our main result is the following.
\begin{theorem}
Let $(M,g(t),\phi(t),p)$ be a complete pointed Type I harmonic Ricci flow on $[0,T)$, $T < \infty$, with non-increasing coupling function $\alpha(t) \in [\underline{\alpha},\overline{\alpha}]$, where $0<\underline{\alpha}\le \overline{\alpha}<\infty$. Then for any given sequence $\lambda_j \nearrow +\infty$, and any point $p \in M$, the sequence $(M,g_j(t),\phi_j(t),p)$ of harmonic Ricci flows, where $g_j(t) \coloneqq \lambda_j g(T+\frac{t}{\lambda_j})$, $\phi_j(t) \coloneqq \phi(T+\frac{t}{\lambda_j})$ on $[-\lambda_j T,0)$, with coupling functions $\alpha_j(t)=\alpha(T+\frac{t}{\lambda_j})$, admits a subsequence converging in the pointed Cheeger-Gromov sense to a normalized gradient shrinking harmonic Ricci soliton solution in canonical form $(M_{\infty},g_{\infty}(t),\phi_{\infty}(t),p_{\infty})$, $t \in (-\infty,0)$, with constant coupling function $\lim_{t\nearrow T}{\alpha(t)}$. Moreover, any subsequential limit is of this form. Finally, if the point $p \in \Sigma_I$ is a Type I singular point (see Definition \ref{definitiontypeIsing}), then the limit is non-trivial, i.e. not locally isometric to the Gaussian soliton.
\label{nontrivial}
\end{theorem}
This theorem generalizes Theorem $4.6$ in \cite{shi}, where Shi supposed that the limit above is a closed manifold, and Theorem $2$ in \cite{guo}, in which the authors considered the case $N=\R$, so List's flow.

In order to prove Theorem \ref{nontrivial}, we choose to adapt the strategy from the Ricci flow case outlined above, that is to introduce a reduced distance based at a singular time for the harmonic Ricci flow and to show a pseudolocality theorem. These two results might be of independent interest. Firstly, let us recall the definition of reduced length based at a regular time.
\begin{definition}[Definition 1.2 in \cite{mul1}]
Let $(g(t),\phi(t))$ be a maximal harmonic Ricci flow defined on $[0,T)$. For any $0\le \bar{t}<t_0 <T$, we define the \textit{$\mathcal{L}$-length} of a curve $\gamma \colon [\bar{t},t_0] \longrightarrow M$ as
\begin{equation}
\mathcal{L}(\gamma) \coloneqq \int_{\bar{t}}^{t_0}{\sqrt{t_0-t}\big(|\dot{\gamma}(t)|^2_{g(t)}+\Sh_{g(t)}(\gamma(t)) \big) dt}.
\end{equation}
For a fixed space-time point $(p,t_0)\in M\times[0,T)$, we also define the functions $l_{p,t_0},L_{p,t_0} \colon M \times (0,t_0) \longrightarrow \R$ as
\begin{equation}
l_{p,t_0}(q,\bar{t}) \coloneqq \inf_{\gamma}{\frac{1}{2 \sqrt{t_0-\bar{t}}} \mathcal{L}(\gamma)} \eqqcolon \frac{1}{2 \sqrt{t_0-\bar{t}}} L_{p,t_0}(q,\bar{t}),
\end{equation}
where $\gamma$ ranges in the set of curves such that $\gamma(\bar{t})=q$ and $\gamma(t_0)=p$. The function $l_{p,t_0}$ is called \textit{reduced length based at $(p,t_0)$}.
\label{definitionreducedlength}
\end{definition}
In order to carry out our intended programme for the harmonic Ricci flow, we need to extend this definition to a singular space-time base point $(p,T)$.
\begin{theorem}
Let $(g(t),\phi(t))$ be a complete Type I harmonic Ricci flow on $[0,T)$, $T<\infty$. Fix $t_i \nearrow T$ and a point $p \in M$. Then there exists a locally Lipschitz function
\begin{equation}
l_{p,T} \colon M \times (0,T) \rightarrow \R,
\end{equation}
which is a subsequential limit of the reduced lengths $l_{p,t_i}$ in $C^{0,1}_{loc}$. Moreover, defining 
\begin{equation}
v_{p,T}(q,\bar{t}) \coloneqq (4\pi(T-\bar{t}))^{-\frac{n}{2}}e^{-l_{p,T}(q,\bar{t})},
\end{equation}
these functions verify in the sense of distributions the following inequalities:
\begin{equation}
\begin{dcases*}
-\frac{\partial l_{p,T}(q,\bar{t})}{\partial \bar{t}}-\Delta l_{p,T}(q,\bar{t})+|\nabla l_{p,T}(q,\bar{t})|^2-\Sh_{g(\bar{t})}+ \frac{n}{2(T-\bar{t})} \ge 0 \iff \square^*v_{p,T}(q,\bar{t}) \le 0\\
-|\nabla l_{p,T}(q,\bar{t})|^2+\Sh_{g(\bar{t})}+ \frac{l_{p,T}(q,\bar{t})-n}{(T-\bar{t})}+2\Delta l_{p,T}(q,\bar{t}) \le 0\\
-2\frac{\partial l_{p,T}(q,\bar{t})}{\partial \bar{t}}+|\nabla l_{p,T}(q,\bar{t})|^2-\Sh_{g(\bar{t})}+\frac{l_{p,T}(q,\bar{t})}{(T-\bar{t})}=0.
\end{dcases*}
\end{equation}
\label{redlensingtime}
\end{theorem}
After obtaining the proper compactness theorems for sequences of harmonic Ricci flows, we show the following pseudolocality theorem: 
\begin{theorem}
For every $\beta>0$, $N$, $\underline{\alpha}$ and $\overline{\alpha}$, where $0<\underline{\alpha} \le \overline{\alpha}<\infty$, there exist $\delta,\e>0$ with the following
property. Suppose that we have a pointed complete harmonic Ricci flow $(M,g(t),\phi(t), p)$ defined for $t \in [0,(\e r_0)^2]$, with non-increasing coupling function $\alpha(t) \in [\underline{\alpha},\overline{\alpha}]$. Suppose the following:
\begin{itemize}
\item $\Sh(0) \ge -r_0^2$ on $B_{g(0)}(p,r_0)$;
\item $\Area_{g(0)}(\partial \Omega)^n \ge (1-\delta)c_n \Vol_{g(0)}(\Omega)^{n-1}$, for any $\Omega \subset B_{g(0)}(p,r_0)$, where $c_n$ is the Euclidean isoperimetric constant.
\end{itemize}
Then $|\Rm|(x,t)< \beta t^{-1} + (\e r_0)^{-2}$ whenever $0<t\le (\e r_0)^2$ and $d_{g(t)}(x,p)\le \e r_0$.
\label{pseudo}
\end{theorem}
This article is organized as follows. In Section \ref{results}, we first recall a few basic facts about the harmonic Ricci flow, then we develop a compactness theory for the flow and we conclude the section introducing the concept of Type I harmonic Ricci flow. In Section \ref{sss} the concepts of self-similar solution and harmonic Ricci soliton are recalled, analyzed and linked. The reduced length and volume based at singular time are studied in Section \ref{srl}, where Theorem \ref{redlensingtime} is proved. We prove Theorem \ref{pseudo} in Section \ref{sp}. Finally, in Section \ref{smt} we bring all these results together to prove our main result, Theorem \ref{nontrivial}.

Throughout the rest of the paper we adopt the convention that any connection symbol $\nabla$,  norm and tensor is computed with respect to the evolving metric $g(t)$ unless otherwise stated. Also, $C$ is a generic constant possibly changing from line to line.
\section*{Acknowledgements}
\noindent I would like to express my gratitude to my supervisor Reto Buzano for introducing me to the subject, for his patient guidance and helpful suggestions.

\section{Preliminary Results on Harmonic Ricci Flow}
\label{results}
In this section, we briefly review some definitions and tools from harmonic Ricci flow theory developed by M\"uller in \cite{mul} before developing distance distortion estimates and compactness theorems, and prove some preliminary results about Type I flows. Regarding all the results in this section, we would like to stress the dominance of the geometric component of the flow over the map one.

\noindent M\"uller computed the evolution equations for curvatures, the tensor $\mathcal{S}$, their derivatives and derivatives of the map along the harmonic Ricci flow (\ref{hrf}), we recall here only the ones relevant for our scope. 
\begin{lemma}[Corollary $2.6$ in \cite{mul}]
Given a solution $(g(t),\phi(t))$ of (\ref{hrf}) we have
\begin{align}
\frac{\partial \Sh}{\partial t} &= \Delta \Sh +2|S_{ij}|^2+2\alpha |\tau_g \phi|^2-\dot{\alpha} |\nabla \phi|^2, \\
\frac{\partial S_{ij}}{\partial t} &=\Delta_L S_{ij}+2\alpha\tau_g \phi \nabla_i \nabla_j \phi-\dot{\alpha} \nabla_i \phi \nabla_j \phi,
\label{Sevolution}
\end{align}
where $\Delta_L$ is the Lichnerowicz Laplacian, defined for symmetric $(0,2)-$tensors $t_{ij}$ by the formula 
\begin{equation*}
\Delta_L t_{ij}\coloneqq \Delta t_{ij}+2R_{ipjq}t_{pq}-R_{ip}t_{pj}-R_{jp}t_{ip}.
\end{equation*}
\end{lemma}
Using a simple maximum principle argument, M\"uller obtained that $\Sh_{min}(t)$ is monotonically increasing, thus $\Sc=\alpha |\nabla \phi|^2+\Sh \ge \alpha(t) |\nabla \phi|^2+\Sh_{min}(0)$. A consequence of this is the following
\begin{lemma}[Corollary $2.8$ in \cite{mul}]
Suppose $(g(t),\phi(t))$ is a solution of (\ref{hrf}), with $\alpha(t) \ge \underline{\alpha}>0$. If there exist $(x_k,t_k) \subset M \times [0,T)$, with $t_k \nearrow T$ and $|\nabla \phi|^2(x_k,t_k) \rightarrow +\infty$, then also $\Sc(x_k,t_k)\rightarrow +\infty$, thus the flow becomes singular at $T_{max} \le T$.
\label{rbounds}
\end{lemma}
For later use, we recall the differential (in-)equalities satisfied by the squared norms of the gradient of $\phi$, the hessian of $\phi$ and the Riemannian tensor. Along the harmonic Ricci flow, $|\nabla \phi|^2$ evolves by
\begin{equation}
\frac{\partial |\nabla \phi|^2}{\partial t} = \Delta |\nabla \phi|^2-2\alpha |\nabla \phi \otimes \nabla \phi|^2-2 |\nabla^2 \phi|^2+2 \langle \Rm^N( \nabla_i \phi,\nabla_j \phi)\nabla_j \phi,\nabla_i \phi) \rangle.
\label{grad}
\end{equation}
Assuming $N$ to be of bounded sectional curvatures, we get the following inequalities, with constants depending only on the dimension $n$ of $M$ and on the presumed curvature bounds on N:
\begin{equation}
\frac{\partial |\Rm|^2}{\partial t} \le \Delta |\Rm|^2-2 |\nabla \Rm|^2+C |\Rm|^3+\alpha C |\nabla \phi|^2 |\Rm|^2 +\alpha C |\nabla^2 \phi|^2 |\Rm| +\alpha C |\nabla \phi|^4 |\Rm|,
\label{riem}
\end{equation}
\begin{equation}
\frac{\partial |\nabla^2 \phi|^2}{\partial t} \le \Delta |\nabla^2 \phi|^2-2 |\nabla^3 \phi|^2+C |\Rm||\nabla^2 \phi|^2+\alpha C |\nabla \phi|^2 |\nabla^2 \phi|^2 +C |\nabla \phi|^4 |\nabla^2 \phi| +C |\nabla \phi|^2 |\nabla^2 \phi|^2.
\label{hess}
\end{equation}
Based on these inequalities, M\"uller developed an iterative scheme which yields a long time existence result.
\begin{theorem}[Theorem $3.12$ in \cite{mul}]
Given a solution $(g(t),\phi(t))$ of (\ref{hrf}) on $[0,T)$, $T<\infty$, with a non-increasing coupling function $\alpha(t) \in [\underline{\alpha},\overline{\alpha}]$, $0<\underline{\alpha}\le \overline{\alpha}<+\infty$. Then $T$ is maximal, i.e. the flow cannot be extended past $T$ if and only if 
\begin{equation}
\limsup_{t \nearrow T}{(\sup_{x\in M}{|\Rm(x,t)|^2})}=+\infty.
\end{equation}
\label{longtime}
\end{theorem}
\subsection{Distance Distortion Estimates}
\label{lengthdistortion}
Denote by $d_t \colon M \times M \longrightarrow \R$ the distance function induced by the Riemannian metric $g(t)$ and by $B_t(x,r)$ the geodesic ball centred at $x$ with radius $r$ induced by it.
\begin{lemma}
Let $(M,g(t),\phi(t))$ be a harmonic Ricci flow with target manifold $N$. Then $\forall t_0<t_1$ such that $\mathcal{S} \le K$ on $[t_0,t_1]$ and every two points $x_0,x_1 \in M$ we have
\begin{equation*}
\frac{d_{t_1}(x_0,x_1)}{d_{t_0}(x_0,x_1)} \ge e^{-K(t_1-t_0)}.
\end{equation*}
\begin{proof}
For a fixed curve $\gamma \colon [0,a] \longrightarrow M$, its length evolves by
\begin{align*}
\frac{d}{dt} \Length(\gamma)=\frac{d}{dt} \int_0^a{\sqrt{\Big< \frac{d\gamma}{ds}(s)\frac{d\gamma}{ds}(s) \Big>_t} ds}= \int_0^a{\mathcal{S} \Big(\frac{d\gamma}{ds}(s),\frac{d\gamma}{ds}(s) \Big) \frac{ds}{\big|\frac{d\gamma}{ds}\big|}} \ge -K \Length(\gamma).
\end{align*}
Integrating this gives 
\begin{equation*}
\frac{\Length(\gamma) \big|_{t_1}}{\Length(\gamma) \big|_{t_0}} \ge e^{-K(t_1-t_0)}
\end{equation*}
so we deduce the conclusion by approximating the distance between $x_0$ and $x_1$ at time $t=t_1$ with $\gamma$.
\end{proof}
\label{lengthdist}
\end{lemma}
\begin{remark}
Similarly, if $\mathcal{S} \ge -K$, we have 
\begin{equation*}
\frac{d_{t_1}(x_0,x_1)}{d_{t_0}(x_0,x_1)} \le e^{K(t_1-t_0)}.
\end{equation*}
\end{remark}
\begin{lemma}
Let $(M,g(t),\phi(t))$ be a \emph{complete} harmonic Ricci flow with target manifold $N$. Suppose that $x_0,x_1 \in M$ are two points with $d_{t_0}(x_0,x_1) \ge 2r_0$ and $\Ric(x,t_0) \le K$ $\forall x \in B_{t_0}(x_0,r_0) \cup B_{t_0}(x_1,r_0)$. Then $\frac{d}{dt}d_t(x_0,x_1) \ge -2(\frac{2}{3}Kr_0+(n-1)r_0^{-1})$ at $t=t_0$.
\begin{proof}
Let $\gamma$ be a normalized minimal geodesic from $x_0$ to $x_1$ in $(M,g(t_0))$, and set $X(s) \coloneqq \frac{d \gamma}{ds}(s)$. Then for every vector field $V \in \mathcal{T}(\gamma)$ along $\gamma$ such that $V(x_0)=V(x_1)=0$, from the second variation of length we have
\begin{equation}
\int_0^{d_{t_0}(x_0,x_1)}{\Big(|\nabla_X V|^2-\scal{\Sc(V,X)V}{X}\Big) ds} \ge 0.
\end{equation}
Extend $X$ to a parallel orthonormal frame along $\gamma$ with $\set{e_i(s)}_{i=1}^{n-1}$. Put $V_i(s)=f(s)e_i(s)$ where 
\begin{equation}
f(s)=
\begin{cases}
\frac{s}{r_0} \ \ \ &\text{if} \ 0 \le s \le r_0,\\
1 \ \ \ \ &\text{if} \ r_0 \le s \le d_{t_0}(x_0,x_1)-r_0,\\
\frac{d_{t_0}(x_0,x_1)-s}{r_0} \ \ \ &\text{if} \ d_{t_0}(x_0,x_1)-r_0 \le s \le d_{t_0}(x_0,x_1).
\end{cases}
\end{equation}
Thus we have $|\nabla_X V_i|=|f'(s)|$ and 
\begin{equation*}
\int_0^{d_{t_0}(x_0,x_1)}{|\nabla_X V_i|^2 ds}=2\int_0^{r_0}{\frac{1}{r_0}}=\frac{2}{r_0}.
\end{equation*}
Next, we compute
\begin{align*}
\int_0^{d_{t_0}(x_0,x_1)}{\scal{\Sc(V_i,X)V_i}{X} ds}=&\int_0^{r_0}{\frac{s^2}{r_0^2}\scal{\Sc(e_i,X)e_i}{X} ds}+\int_{r_0}^{d_{t_0}(x_0,x_1)-r_0}{\scal{\Sc(e_i,X)e_i}{X} ds}\\
&+\int_{d_{t_0}(x_0,x_1)-r_0}^{d_{t_0}(x_0,x_1)}{\frac{(d_{t_0}(x_0,x_1)-s)^2}{r_0^2}\scal{\Sc(e_i,X)e_i}{X} ds}
\end{align*}
so summing over $i$ we get
\begin{align*}
0 \le& \sum_{i=1}^{n-1} \int_0^{d_{t_0}(x_0,x_1)}{\big( |\nabla_X V_i|^2-\scal{\Sc(V_i,X)V_i}{X} \big) ds}=\frac{2(n-1)}{r_0}-\int_0^{d_{t_0}(x_0,x_1)}{\Ric(X,X) ds}\\
&-\int_0^{r_0}{\bigg(\frac{s^2}{r_0^2}-1 \bigg) \Ric(X,X) ds}-\int_{d_{t_0}(x_0,x_1)-r_0}^{d_{t_0}(x_0,x_1)}{\bigg(\frac{(d_{t_0}(x_0,x_1)-s)^2}{r_0^2}-1 \bigg) \Ric(X,X) ds}.
\end{align*}
Taking a geodesic variation of $\gamma$ with fixed endpoints for $t$ near $t_0$ (using the completeness assumption) we have
\begin{align*}
\frac{d}{dt} d_t(x_0,x_1) \big|_{t=t_0}&=-\int_0^{d_{t_0}(x_0,x_1)}{\mathcal{S}(X,X) ds} \ge -\int_0^{d_{t_0}(x_0,x_1)}{\Ric(X,X) ds} \\
&\ge -\frac{2(n-1)}{r_0} -2K\frac{2}{3}r_0. \qedhere
\end{align*}
\end{proof}
\label{lengthdist1}
\end{lemma}
\begin{corollary}
If $\Ric \le K$ in a \emph{complete} harmonic Ricci flow then there exists a constant $c=c(n)>0$ such that for every $x_0,x_1 \in M$, $\frac{d}{dt} d_t(x_0,x_1) \ge -cK^{1/2}$.
\begin{proof}
Put $r_0 =K^{-\frac{1}{2}}$. If $d_t(x_0,x_1) \le 2r_0$ then use Lemma \ref{lengthdist}, otherwise use Lemma \ref{lengthdist1}.
\end{proof}
\end{corollary}
With a very similar proof as the Lemma \ref{lengthdist1} one can prove (see Lemma $4$ in \cite{guo}):
\begin{lemma}
Let $(M,g(t),\phi(t))$ be a \emph{complete} harmonic Ricci flow with target manifold $N$.

\noindent If $\Ric(x,t_0) \le K$ $\forall x \in B_{t_0}(x_0,r_0)$ then for every $x \notin B_{t_0}(x_0,r_0)$ we have
\begin{equation*}
\frac{d}{dt} d_t(x_0,x) \bigg|_{t=t_0} -\Delta_{t_0} d_{t_0}(x_0,x) \ge -\bigg(\frac{2}{3}Kr_0+(n-1)r_0^{-1}\bigg).
\end{equation*} 
\label{lengthdist2}
\end{lemma}
\subsection{Compactness Theorems}
\label{compactness}
In this section we provide several refinements of the compactness result for harmonic Ricci flows obtained by Shi in \cite{shi} and Williams in \cite{wil}. One of these improvements is needed in the proof of the pseudo-locality theorem where only local bounds on the curvature are available, and therefore Shi's and Williams' results do not directly apply.

We follow the argument developed by Topping in \cite{top1} for the Ricci flow case (extending Hamilton's compactness result in \cite{ham2}), see also \cite{top2} for an expository review. 
The first theorem we want to prove is the following.
\begin{theorem}[Compactness of harmonic Ricci flows: Extension 1]
Let $(M_i,g_i(t),\phi_i(t),p_i)$ be a given sequence of $n-$dimensional, pointed, complete harmonic Ricci flows, with fixed target manifold $N$, all defined for $t \in (a,b)$, where $-\infty \le a<0<b \le +\infty$, with (possibly different) non-increasing coupling functions $\alpha_i(t) \in [\underline{\alpha},\overline{\alpha}]$, where $0<\underline{\alpha} \le \overline{\alpha}<\infty$. Assume further that
\begin{itemize}
\item[(i)] $\forall r>0$ there exists a constant $M=M(r)$ such that $\forall t \in (a,b)$ and $\forall i$ 
\begin{equation*}
\sup_{B_{g_i(0)}(p_i,r)}{|\Rm(g_i(t))|_{g_i(t)}}\le M,
\end{equation*}
\item[(ii)] $\inf_i{\inj(M_i,g_i(0),p_i)}>0$.
\end{itemize}
Then there exist an $n-$dimensional manifold $M_{\infty}$, a harmonic Ricci flow $(g_{\infty}(t),\phi_{\infty}(t))$ on $M_{\infty} \times (a,b)$ with target manifold $N$ and non-increasing coupling function $\alpha_{\infty}(t)$, and a point $p_{\infty}\in M_{\infty}$, such that $(M_{\infty},g_{\infty}(0))$ is complete, and we have $(M_i,g_i(t),\phi_i(t),p_i) \longrightarrow (M_{\infty},g_{\infty}(t),\phi_{\infty}(t),p_{\infty})$ in the pointed Cheeger-Gromov sense (after passing to a subsequence) and $\alpha_i(t)\longrightarrow \alpha_{\infty}(t)$ pointwise in $(a,b)$.
\label{compact1}
\end{theorem}
Before starting with the proof, we need to recall some results.
\begin{definition}
A sequence of $n-$dimensional, pointed, harmonic Ricci flows $(M_i,g_i(t),\phi_i(t),p_i)$, with fixed target manifold $N$, defined for $t \in (a',b')$, where $-\infty \le a'<0<b' \le +\infty$, is said to satisfy \textit{uniform derivative curvature bounds} if $\forall s>0$, $\forall k \in \mathbb{N}$, there exist constants $M(k,s)$ and $C(s)$ such that $\forall t \in (a',b')$ and $\forall i$ we have 
\begin{equation}
\sup_{B_{g_i(0)}(p_i,s)}{|\nabla_i \phi_i|_{g_i}(t)}\le C(s), \sup_{B_{g_i(0)}(p_i,s)}{ \Big( |\nabla_i^{(k)} \Rm_i|_{g_i}(t)+|\nabla_i^{(k+2)} \phi_i|_{g_i}(t) \Big)}\le M(k,s).
\label{uniformderivativecurvaturebounds}
\end{equation}
\end{definition}
Recall the following proposition from \cite{mul}:
\begin{proposition}[Proposition $A.5$, \cite{mul}]
Let $(g(t),\phi(t))$ be a complete solution of the harmonic Ricci flow defined on $[0,T)$, with a non-increasing coupling function $\alpha(t) \in[\underline{\alpha},\overline{\alpha}]$, where $0<\underline{\alpha} \le \overline{\alpha}< \infty$ and let $T' <T<+ \infty$. Define $B \coloneqq B_{g(T')}(x,r)$, and assume that $|\Rm| \le R_0$ on $B \times [0,T']$ for some constant $R_0$. Then there exist constants $K=K(\underline{\alpha},\overline{\alpha},R_0,T,n,N)$ and $C_k=C_k(k,\overline{\alpha},n,N)$ for $k \in \mathbb{N}$, $C_0=1$, such that for all $k \ge 0$ we have
\begin{equation}
|\nabla \phi|^2\le \frac{K}{t},|\Rm|\le \frac{K}{t}, |\nabla^k \Rm|^2+|\nabla^{k+2} \phi|^2 \le C_k \bigg( \frac{K}{t}\bigg)^{k+2}
\end{equation}
for every $(x,t) \in B_{g(T')}(x,r/2) \times (0,T')$.
\end{proposition}
\begin{remark}
We first remark that this theorem gives local estimates depending only on the curvature bound and on the (possibly singular) time $T$. In the case $T$ is not singular (which is the case we will need), we can take $T'=T$.

Secondly, given the curvature bound $R_0$, we know $B_{g(0)}(x,\sigma_1 r) \subset B_{g(T')}(x,r) \subset B_{g(0)}(x,\sigma_2 r)$ for constants $\sigma_1,\sigma_2$ which depend only on $R_0$ and on the finiteness of the time $T$ (see Lemma \ref{lengthdist}), thus we can assume that both the bounds in the hypothesis and the conclusion hold on time $(t=0)-$balls crossed with the time interval $(0,T)$.
\end{remark}
\begin{remark}
As Hamilton remarked in \cite{ham2} for the corresponding Ricci flow result, it is enough to prove Theorem \ref{compact1} in the case $a,b \in \R$. Once proved on finite intervals, a standard diagonal argument yields the other cases. Furthermore, from the assumption $(i)$ in the main Theorem \ref{compact1}, we get that the sequence verifies \textit{uniform derivative curvature bounds} in a slightly smaller finite interval $[a+\e,b-\e]$ for any arbitrarily small $\e>0$ from the proposition above. So if we can prove the theorem with the extra assumption of having \textit{uniform derivative curvature bounds}, we can take a sequence $\e_j$ tending to $0$, and for any such $j$ we have a limit harmonic Ricci flow that agrees with the ones obtained for smaller $j$ (on the interval where both limits exist), and so, by a diagonal argument, we obtain a limit flow on the interval $(a,b)$ as required.
\end{remark}
The plan of the proof is the same as in \cite{ham2}: we first extract a limit manifold of the $t=0$ time-slices, and then use uniform bounds to obtain convergence at the other times. The compactness theorem of manifolds we will use is the following (cfr. Theorem 2.5 in \cite{top2}):
\begin{theorem}[Compactness of manifolds]
Suppose that $(M_i,g_i,p_i)$ is a sequence of (smooth) complete, pointed Riemannian manifolds, all of dimension $n$, satisfying:
\begin{itemize}
\item[(i)] $\forall r>0$ and $\forall k \in \mathbb{N}$ we have 
\begin{equation*}
\sup_i \sup_{B_{g_i}(p_i,r)}{|\nabla_i^k \Rm(g_i)|_{g_i}}<\infty
\end{equation*}
and
\item[(ii)] $\inf_i{\inj(M_i,g_i,p_i)}>0$.
\end{itemize}
Then there exists a (smooth) complete, pointed Riemannian manifold $(M_{\infty},g_{\infty},p_{\infty}$) of dimension $n$ such that, after passing to some subsequence, $(M_i,g_i,p_i) \longrightarrow (M_{\infty},g_{\infty},p_{\infty})$ in the pointed Cheeger-Gromov sense.
\label{compactnessmanifold}
\end{theorem}
The uniform bounds we will need are given by the following (Lemma 2.10 in \cite{shi} or Lemma 3.5 in \cite{wil}):
\begin{lemma}
Let $(M, h)$ be a Riemannian manifold, $K \subset \subset M$ and $(g_i(t), \phi_i(t))$ a sequence of harmonic Ricci flows defined on a neighbourhood of $K \times [a',b']$ with $[a',b']$ containing $0$. Suppose that for each $k \in \mathbb{N}$, we have the following conditions:
\begin{itemize}
\item[(a)] $C^{-1}h \le g_i(0)\le C h$ on $K$ for every $i$,
\item[(b)] $|\nabla^k g_i(0)|_h+|\nabla^k \phi_i(0)|_h \le C_k$ on $K$, for every $i$,
\item[(c)] $|\nabla_i \phi_i|_i\le C',|\nabla_i^{(k)} \Rm_i|_i+|\nabla_i^{(k+2)} \phi_i|_i \le C'_k$ on $ K \times [a',b']$,
\end{itemize}
where the constants are independent of $i$, but may depend on $K, a', b'$ and $k$. Then for all $k,s \in \mathbb{N}$ there exist constants $\tilde{C}, \tilde{C}_{k,s}$ independent of $i$ such that
\begin{itemize}
\item $\tilde{C}^{-1}h \le g_i(t)\le \tilde{C} h$ on $K \times [a',b']$ for every $i$,
\item We have on $K \times [a',b']$ for every $i$
\begin{equation*}
\Big| \frac{\partial^s}{\partial t^s}\nabla^k g_i \Big|_h+\Big|\frac{\partial^s}{\partial t^s}\nabla^k \phi_i \Big|_h \le \tilde{C}_{k,s}
\end{equation*}
\end{itemize}
\label{timesgluing}
\end{lemma}
\begin{proof}[Proof of Theorem \ref{compact1}]
Without loss of generality, we can assume (possibly after extracting a subsequence, still denoted by $i$) that the coupling functions $\alpha_i(t)$ converge pointwise in $(a,b)$ to a certain coupling function $\alpha_{\infty}(t)$, which is still non-increasing and verifies the same bounds as the $\alpha_i$'s.

\noindent We can assume that our sequence satisfies \textit{uniform derivative curvature bounds} by the remarks above. Therefore by Theorem \ref{compactnessmanifold} we can extract a subsequence such that $(M_i,g_i(0),p_i) \longrightarrow (M_{\infty},g_{\infty}(0),p_{\infty})$ in the pointed Cheeger-Gromov sense, for a certain pointed manifold $(M_{\infty},p_{\infty})$ and metric $g_{\infty}(0)$. Let us remark that $(M_{\infty},g_{\infty}(0))$ is complete by the theorem. Call $\Psi_i$ the diffeomorphisms given by the Cheeger-Gromov convergence. Since we have $\sup_{B_{g_i(0)}(p_i,s)}{|\nabla_i^k \phi_i|_{g_i}(0)}\le M(k,s)$, extracting a further subsequence we can assume $\Psi_i^* \phi_i(0) \longrightarrow \phi_{\infty}(0)$ in $C^{\infty}_{loc}(M_{\infty};N)$ for a certain map $\phi_{\infty}(0)$. To extend the convergence at the other times, it is sufficient to apply the Lemma \ref{timesgluing} to $\Psi_i^* g_i(t)$ and $\Psi_i^* \phi_i(t)$, so by Arzel\'{a}-Ascoli's theorem we get a limit flow $g_{\infty}(t)$ for $t \in (a,b)$ and a limit $\phi_{\infty}(t)$ of the maps $\Psi_i^* \phi_i(t)$ in $C^{\infty}_{loc}(M_{\infty}\times (a,b);N)$, which agree with the metric $g_{\infty}(0)$  and the map $\phi_{\infty}(0)$ at $t=0$ respectively, and form a harmonic Ricci flow.
\end{proof}
Once we have a limit harmonic Ricci flow, we can inquire its completeness. As remarked during the proof of the main Theorem \ref{compact1}, we already have completeness at the $(t=0)-$time-slice. A good idea to get completeness at other time-slices is to exploit the length distortion estimates obtained in the Section \ref{lengthdistortion} together with the Hopf-Rinow Theorem, as in \cite{top1}.
\begin{theorem}[Compactness of harmonic Ricci flows: Extension 2]
Under the same assumptions as in Theorem \ref{compact1}, suppose instead of having $(i)$ we have the stronger assumption that
\begin{itemize}
\item[(i')] There exists $M<\infty$ with the following property: $\forall r>0$ there exists $K=K(r) \in \mathbb{N}$ such that $\forall t \in (a,b)$ and $i \ge K$ we have
\begin{equation}
\sup_{B_{g_i(0)}(p_i,r)}{|\Rm(g_i(t))|_{g_i(t)}}\le M.
\end{equation}
\end{itemize}
Then the harmonic Ricci flow constructed in Theorem \ref{compact1} has complete time-slices.
\begin{proof}
Since $(M_{\infty},g_{\infty}(0))$ is complete, for every $r>0$ we have $B_{g_{\infty}(0)}(p_{\infty},\frac{r}{2}) \subset \subset M_{\infty}$. By the assumption $(i')$ and the convergence, we have $\sup_{B_{g_{\infty}(0)}(p_\infty,\frac{r}{2})}{|\Rm(g_{\infty}(t))|_{g_{\infty}(t)}}\le M$. Since $M$ is independent of $r$, the limit harmonic Ricci flow has bounded curvature at every time-slice. Using the length distortion estimate, for every $s>0$ we have $B_{g_{\infty}(t)}(p_\infty,s) \subseteq B_{g_{\infty}(0)}(p_\infty,C(M,t)s)\subset \subset M_{\infty}$.
\end{proof}
\label{compact2}
\end{theorem}
\begin{remark}
We remark that this version of the compactness theorem will be used in the proof of the pseudolocality theorem.
\end{remark}
In the theorem above we used the strong bound on the curvature to apply the length distortion estimate in Lemma \ref{lengthdist}. Since the hypothesis of the latter are weaker, we can improve the result, showing that a uniform unilateral bound on the tensors $\mathcal{S}_{g_i(t)}$ implies completeness in the past or the future.
\begin{theorem}[Compactness of harmonic Ricci flows: Extension 3]
Under the same assumptions as in Theorem \ref{compact1}, if there exists a constant $C>0$ such that
\begin{equation*}
\inf_i \inf_{M_i}{\mathcal{S}_{g_i(t)}} \ge -C \ \ \ \forall t \in (a,0]
\end{equation*}
then $(M_{\infty},g_{\infty}(t))$ is complete $\forall t \in (a,0]$. Analogously, the bound
\begin{equation*}
\sup_i \sup_{M_i}{\mathcal{S}_{g_i(t)}} \le C \ \ \ \forall t \in [0,b)
\end{equation*}
 implies completeness of time-slices for $t \in [0,b)$.
\end{theorem}
\subsection{Type I and Type A Conditions}
\label{typeItypeA}
In this section we define what is a Type I singularity for the harmonic Ricci flow, and what is a Type I harmonic Ricci flow.
It is worth mentioning that in \cite{shi}, Shi gave a definition of Type I singularity expressed in terms of the so-called \textit{AC Curvature} defined as $Q(x,t) \coloneqq (|\Rm|+ |\nabla^2 \phi|+|\nabla \phi|^2)(x,t)$. In particular, he proved that if a harmonic Ricci flow develops a singularity at a time $T<+\infty$, then there exists a constant $c>0$ such that
\begin{equation}
\sup_{x \in M}{Q(x,t)} \ge \frac{c}{T-t}.
\end{equation}
However, motivated by Theorem \ref{longtime}, we would like a definition focusing only on the metric component. Therefore our definition will be different from this and coherent with the heuristic dominance of the metric component over the map one. Throughout this subsection we assume for simplicity that the domain manifold $M$ is closed, though we only need to use the strong maximum principle.
\begin{lemma}
Let $(g(t),\phi(t))$ be a solution of (\ref{hrf}) on $[0,T)$ with a non-increasing coupling function $\alpha(t) \in [\underline{\alpha},\overline{\alpha}]$, $0<\underline{\alpha}\le \overline{\alpha}<+\infty$. Suppose $T<+\infty$ is maximally chosen. Then there exists a constant $c>0$ and a sequence $(t_k)\nearrow T$ such that 
\begin{equation}
\sup_{x \in M}{|\Rm|(x,t_k)} \ge \frac{c}{T-t_k}.
\label{leastrate}
\end{equation}
\begin{proof}
Let us set 
\begin{equation}
y(t)\coloneqq \max_{x \in M}{|\nabla \phi|^2(x,t)}, \ w(t)\coloneqq \max_{x \in M}{|\nabla^2 \phi|^2(x,t)} \ and \ z(t)\coloneqq \max_{x \in M}{|\Rm|^2(x,t)}.
\label{definitionywz}
\end{equation}
The proof is split in two cases. Firstly, if we have that $y$ is unbounded, an application of the maximum principle to the equation (\ref{grad}) yields $y\ge \frac{c}{T-t}$, thus using (\ref{Sevolution}) as in the proof of Lemma \ref{rbounds} we get the desired bound. In the case $y$ is bounded, we can absorb it in the constants appearing in (\ref{riem}) and (\ref{hess}). Writing the differential inequalities for $\sqrt{w}$ and $\sqrt{z}$ and summing them up, we obtain again via a maximum principle argument that $2 \cdot max \{ z(t),w(t) \}\ge (z+w)(t) \ge \frac{c_1}{(T-t)^2}$. Proceeding now by contradiction, suppose that for every $\e>0$ there exists a $T(\e)$ such that $z(t) \le \frac{\e}{T-t}$ for $t \in [T(\e),T)$. Then rewrite (\ref{hess}) with this new bound to get:
\begin{equation}
\frac{\partial w}{\partial t} \le c \sqrt{z} w+cw \le c \sqrt{z} w \le \frac{c \e}{T-t} w, \ \ \ in \ \ [T(\e),T).
\end{equation}
One can apply Gronwall's Lemma to get $w(t) \le \frac{c}{(T-t)^\e}$ which is contradictory for $\e< \min \{ c_1,2 \}$.
\end{proof}
\end{lemma}
Therefore we make the following definition.
\begin{definition}
We say that a harmonic Ricci flow $(g(t),\phi(t))$ develops a \textit{Type I singularity} at the finite time $T$ if there exist a constant  $c>0$ and a sequence $(t_k) \nearrow T$ such that 
\begin{equation}
\sup_{x \in M}{|\Rm|(x,t_k)} \sim \frac{c}{T-t_k}.
\end{equation}
\label{typeIsingdef}
\end{definition}
We are naturally tempted to define a Type I harmonic Ricci flow by ``reversing" the inequality (\ref{leastrate}); since our main scope will be to perform blow up arguments, we need to control all the components of the flow. Fortunately we have:
\begin{theorem}
Let $(g(t),\phi(t))$ be a solution of (\ref{hrf}) on $[0,T)$ with a non-increasing coupling function $\alpha(t) \in [\underline{\alpha},\overline{\alpha}]$, $0<\underline{\alpha}\le \overline{\alpha}<+\infty$. Suppose that there exist constants $r\ge 1$ and $C>0$ such that
\begin{equation}
|\Rm|(x,t) \le \frac{C}{(T-t)^r}.
\end{equation}
Then the same inequality, with a different constant $\tilde{C}=\tilde{C}(\underline{\alpha},\overline{\alpha},n,N,C,r)$ holds for $|\nabla \phi|^2$ and $|\nabla^2 \phi|$. Moreover
\begin{equation}
|\nabla \Rm|(x,t) \le \frac{C}{(T-t)^{\frac{3}{2} r}}.
\end{equation}
\begin{proof}
The proof is straight-forward and we therefore only give a brief sketch. Using the bound on the Riemann tensor, we get the same kind of bound for the scalar curvature, thus for $\Sh$ since $\Sh(x,t) \le \Sc(x,t)$. Since $\Sh_{min}(t)$ is increasing, we get the same bound on $|\nabla \phi|^2$. Using the bounds on $y$ and $z$ defined as in (\ref{definitionywz}) to rewrite (\ref{hess}), we get the same bound on $|\nabla^2 \phi|$ by applying the maximum principle to the function $f(x,t)\coloneqq (T-t)^{2r}|\nabla^2 \phi|^2(A+(T-t)^r|\nabla \phi|^2)$ for a large enough constant $A$. This proves the first statement. A similar argument leads to the second statement.
\end{proof}
\label{typerhrf}
\end{theorem}
\begin{definition}
We say that a solution $(g(t),\phi(t))$ of (\ref{hrf}) on $[0,T)$ is a \textit{Type I harmonic Ricci flow} if there exists a constant $C>0$ such that
\begin{equation}
|\Rm|(x,t) \le \frac{C}{(T-t)}.
\label{definitiontypeI}
\end{equation}
We say that it is a \textit{Type A harmonic Ricci flow} if there exist constants $r \in [1,\frac{3}{2})$ and $C>0$ such that
\begin{equation}
|\Rm|(x,t) \le \frac{C}{(T-t)^r}.
\end{equation}
\end{definition}
A priori it might be that the supremum of the curvature oscillates between different rates, e.g. the Type I rate and a smaller one. We will exclude this phenomenon in the Type I flow case, see Corollary \ref{nonoscillation}, where we will see that every singularity in a harmonic Ricci flow satisfying (\ref{definitiontypeI}) is necessarily a Type I singularity as in Definition \ref{typeIsingdef}.
\section{Self-Similar Solutions}
\label{sss}
In this section we briefly recall the definitions of self-similar solutions for the harmonic Ricci flow as well as some of their properties following \cite{mul}. Then we will prove a Zhang-type theorem, generalising results from \cite{guo} and \cite{wil}.
\begin{definition}
Given a coupling function $\alpha(t)$, we say that a solution $(g(t),\phi(t))$ of (\ref{hrf}) defined on $[0,T)$ is a \textit{soliton solution} if there exist a family of diffeomorphisms $\psi_t \colon M \longrightarrow M$ with $t \in [0,T)$, and a  function $c \colon [0,T) \longrightarrow \R^{+}$ such that 
\begin{equation}
\begin{cases*}
g(t)=c(t)\psi_t^*g(0),\\
\phi(t)=\psi_t^* \phi(0).
\end{cases*}
\label{solitongeneral}
\end{equation}
If the derivative of $c$ verifies one of $dc/dt=\dot{c}<0$, $\dot{c}=0$ or $\dot{c}>0$, the solution is called \textit{shrinking, steady or expanding solution} respectively. If the diffeomorphisms $\psi_t$ generate a family of vector fields $X(t)$ with $X(t)=\nabla^{g(t)} f(t)$ for some function $f(t)$ on $M$, we call the solution a \textit{gradient} soliton solution, and the function $f(t)$ is called a \textit{potential} of the soliton solution.
\end{definition}
Self-similar solutions naturally arise from the symmetry properties of the harmonic Ricci flow equations. One important example is the so-called \textit{Gaussian soliton}: consider the time independent flat solution given by $(M,g(t),\phi(t),N) \equiv (\R^n,g_{euc},y_0,N)$, where $g_{euc}$ is the standard Euclidean metric on $\R^n$, $N$ is any target manifold, and $y_0$ is a constant map from $\R^n$ to $N$ with image given by $y_0$. For any coupling function $\alpha(t)$, we can consider this solution as a gradient shrinking solution, defining $g_{euc}=g(t)=c(t)\psi_t^*g_{euc}$, $\phi(t)=\psi_t^* y_0 \equiv y_0$, where $c(t)$ is any decreasing function with $c(0)=1$ and $\psi_t$ is the family of diffeomorphisms generated by the complete vector field $\nabla f(t)$ with $f(t)=\frac{-\dot{c}(t)}{4c(t)}|x|^2$. In the particular case in which $c(t)=T-t$ for some $T$, we will call it \textit{Gaussian soliton in canonical form}. The importance of the Gaussian soliton resides in the fact that it models any harmonic Ricci flow near regular space-time points.

Now we recall the following result which describes a system of elliptic equations solved by a gradient soliton, written in a different manner.
\begin{lemma}[Lemma $2.2$ in \cite{mul}]
Let $(g(t),\phi(t))$ be a gradient soliton with potential $f(t)$ defined on $[0,T)$. Then for any $t \in [0,T)$, the solution satisfies
\begin{equation}
\begin{cases}
\Ric_{g(t)}-\alpha(t) \nabla \phi(t) \otimes \nabla \phi(t)+\Hess(f(t))+ \sigma(t) g(t)=0,\\
\tau_{g(t)} \phi(t)-\scal{\nabla \phi(t)}{\nabla f(t)}=0,
\end{cases}
\label{solitonelliptic}
\end{equation}
where $\sigma(t)=\frac{\dot{c}(t)}{2c(t)}$. Conversely, fix smooth functions $f$ on $M \times [0,T)$, $\alpha$ and $\sigma$ on $[0,T)$. Given a smooth solution $(g(t),\phi(t))$ of (\ref{solitonelliptic}) for any $t$ (also smooth in $t$), suppose that $\nabla f(t)$ is a complete vector field with respect to $g(t)$ for any $t \in [0,T)$. Then we can write the solution $(g(t),\phi(t))$ as a gradient soliton solution to the harmonic Ricci flow with potential $f(t)$, coupling function $\alpha(t)$, and with $c(t)=\exp(2\Sigma(t))$, where $\Sigma(t)=\int_0^t{\sigma(s) ds}$. 
\label{solutionequivsoliton}
\end{lemma}
\begin{remark}
The converse direction of this lemma was stated differently in \cite{mul}, namely as follows: given a function $f$ and a solution of (\ref{solitonelliptic}) at time $t=0$, then there exists a family of diffeomorphisms $\psi_t$, $\psi_0=id$, such that if we define $(g(t),\phi(t))$ as in (\ref{solitongeneral}), with linear scaling function $c(t)=1+2\sigma(0)t$, then  $(g(t),\phi(t))$ solves the harmonic Ricci flow with (constant) coupling function $\alpha=\alpha(0)$. Clearly, this defines a soliton solution. Our version allows non-constant $\alpha$ and $\sigma$.
\end{remark}
\begin{definition}
We say that a gradient soliton solution is \textit{in canonical form} if the scaling function $c(t)$ is linear, in which case (\ref{solitonelliptic}) becomes
\begin{equation}
\begin{cases}
\Ric_{g(t)}-\alpha(t) \nabla \phi(t) \otimes \nabla \phi(t)+\Hess(f(t))+ \frac{a}{2(T-t)} g(t)=0,\\
\tau_{g(t)} \phi(t)-\scal{\nabla \phi(t)}{\nabla f(t)}=0,
\end{cases}
\label{solitoncanonical}
\end{equation}
where (possibly after scaling) $a=+1,0,-1$ respectively in the expanding, steady and shrinking case.
\end{definition}
Differently from what happens in the Ricci flow case, it is easy to construct a soliton solution which is not isometric to one in canonical form; for instance, take $M=N=S^2$, $\alpha(t)=1-t$, $g(t)=(1-t^2)g_S$, $\phi(t)=id_{S^2}$ where $g_S$ is the multiple of the standard metric with (constant) scalar curvature $2$.
\begin{definition}
Given constants $\alpha$ and $\sigma$ and a function $f$ on $M$, then a solution $(g,\phi)$ of 
\begin{equation}
\begin{cases}
\Ric_{g}-\alpha \nabla \phi \otimes \nabla \phi+\Hess(f)+ \sigma g=0,\\
\tau_{g} \phi-\scal{\nabla \phi}{\nabla f}=0
\end{cases}
\label{solitondefinition}
\end{equation}
is called \textit{gradient soliton}. The case $\sigma<0,\sigma=0,\sigma>0$ are called respectively \textit{shrinking, steady} and \textit{expanding} soliton. The function $f$ is called a \textit{potential} of the soliton.
\end{definition}
Lemma \ref{solutionequivsoliton} gives the relation between gradient soliton solutions and (complete) solitons. Recall the following simple consequence of the soliton equations.
\begin{proposition}[Section 2 in \cite{mul}]
Let $(g,\phi)$ be a gradient soliton, i.e. a solution of (\ref{solitondefinition}). Then we have
\begin{equation}
\begin{cases}
\Sc-\alpha |\nabla \phi|^2+\Delta f + \sigma n=0\\
\Sc-\alpha |\nabla \phi|^2+|\nabla f|^2+2 \sigma f=constant.
\end{cases}
\end{equation}
\label{solitonstrange}
\end{proposition}
\begin{definition}
We call a gradient soliton \textit{normalized}, if the constant in the proposition above is $0$. We call a gradient soliton solution \textit{normalized} if its corresponding solitons are normalized at any time-slice.
\end{definition}
Any gradient soliton can be normalized by modifying its potential by a constant.

\subsection{Zhang Type Theorem}
\label{zhangtype}
Motivated by Zhang's paper \cite{zha}, we prove the following theorem. It is worth mentioning that in \cite{guo} the authors proved the first part of this result for List's flow, and our argument closely follows theirs. In addition, we also include a rigidity remark.
\begin{theorem}
Suppose  $(g(t),\phi(t),f(t))$ is a complete gradient shrinking soliton solution. Then we have $\Sh(t) \ge 0$ and $\nabla f(t)$ is a complete vector field at any time-slice. Moreover, if there exist $p \in M$ and $t \in [0,T)$ such that $\Sh(p,t)=0$, then $(M,g(t),\phi(t),f(t))$ is isometric to the Gaussian soliton.
\label{rigidity}
\end{theorem}
Before proceeding with the proof we state a lemma, which is easy to prove using Proposition \ref{solitonstrange} and the Bochner identity (Section 4 in \cite{mul}) and whose proof is therefore left to the reader.
\begin{lemma}
For a normalized gradient shrinking soliton $(g,\phi,f)$ we have
\begin{equation}
\Delta \Sh - \langle \nabla \Sh, \nabla f \rangle=-2\sigma \Sh-2|S_{ij}|^2-2 \alpha (\tau_g \phi)^2.
\label{ellipticsoliton}
\end{equation}
\label{ellipticsolitonlemma}
\end{lemma}
\begin{proof}[Proof of Theorem \ref{rigidity}]
For the first part of the theorem we will work in a fixed time-slice, so without loss of generality we can suppose to have $(g,\phi,f)$ satysfying (\ref{solitondefinition}) with $\sigma=-1/2$. Rewrite Lemma \ref{lengthdist2} in the following way: (for any fixed time slice), if $\Ric \le K$ in the ball $B(p,r_0)$, then
\begin{equation*}
\Delta d(p, \cdot) \le \bigg((n-1)r_0^{-1}+\frac{2}{3}Kr_0\bigg)-\int_0^{d(p,\cdot)}{\Ric(\dot{\gamma}(s),\dot{\gamma}(s)) ds},
\end{equation*}
where $\gamma$ is a normalized geodesic starting at $p$. Applying the first of the soliton equations to the couple $(\dot{\gamma},\dot{\gamma})$ we get $\mathcal{S}(\dot{\gamma},\dot{\gamma})+\nabla^2 f(\dot{\gamma},\dot{\gamma})=1/2$. Moreover, we have
\begin{align*}
\frac{d}{ds} f(\gamma(s))&=\langle \nabla f,\dot{\gamma} \rangle,\\
\frac{d^2}{ds^2} f(\gamma(s))&=\langle \nabla_{\dot{\gamma}} \nabla f,\dot{\gamma} \rangle=\nabla^2 f(\dot{\gamma},\dot{\gamma}),
\end{align*}
therefore for every $x \in M$
\begin{align*}
\int_0^{d(p,x)}{\Ric(\dot{\gamma}(s),\dot{\gamma}(s)) ds}&= \frac{1}{2}d(p,x)-\int_0^{d(p,x)}{\frac{d^2}{ds^2} f(\gamma(s)) ds}+ \alpha \int_0^{d(p,x)}{d\phi \otimes d \phi(\dot{\gamma},\dot{\gamma}) ds}\\
&\ge \frac{1}{2}d(p,x)- \langle \nabla f,\dot{\gamma} \rangle \Big|_{s=0}^{s=d(p,x)}
\ge  \frac{1}{2}d(p,x)- \langle \nabla f(x),\nabla_x d(p,x) \rangle-|\nabla f(p)|.
\end{align*}
In other words, for any fixed point $p \in M$ the function $d(x) \coloneqq d(p,x)$ satisfies
\begin{equation}
\Delta d- \langle \nabla f,\nabla d \rangle \le \bigg((n-1)r_0^{-1}+\frac{2}{3}Kr_0\bigg)-\frac{1}{2}d+|\nabla f(p)|.
\end{equation}
For every point $p$ there exists a small enough radius $r_0=r_0(p)>0$ such that $\Ric(g) \le (n-1)r_0^{-2}$ on the ball $B(p,r_0)$. Fix a cut-off function $\psi(x)$ which is equal to $1$ if $|x|\le 1$ and zero for $|x|>2$, is non-increasing on the positive axis and such that $|\psi'|^2/\psi\le 4$, $|\psi''|,|\psi'|\le 2$. Define then the function on $M$
\begin{equation*}
\eta(x) \coloneqq \psi \bigg(\frac{d(p,x)}{Ar_0} \bigg),
\end{equation*}
where $A$ is a large constant. Define $u\coloneqq \Sh \eta$. It easily follows that
\begin{equation}
\Delta u=\eta \Delta \Sh+\frac{\Sh \psi'}{Ar_0} \Delta d+\frac{\Sh \psi''}{(Ar_0)^2}+2\langle \nabla \Sh,\nabla \eta \rangle.
\end{equation}
Now if the minimum of $u$, which exists because it is a function of compact support, is non negative, then $\Sh(p) \ge 0$ and we have the first assertion since the point $p$ can be chosen arbitrarily. Therefore it is enough to consider the case in which this minimum is strictly negative. Thus the point of minimum $p_{min}$ must be in the support of $\eta$, so it must be in $B(p,2Ar_0)$. We have $\nabla u=0$, $\Delta u \ge 0$, so $\nabla \Sh= \frac{-\Sh \nabla \eta}{\eta}$. Clearly, we also have $u \psi' \ge 0$. Using Lemma \ref{ellipticsolitonlemma} above, we get
\begin{align*}
\Delta u&=\eta(\langle \nabla \Sh, \nabla f \rangle+\Sh-2|S_{ij}|^2-2 \alpha (\tau_g \phi)^2)+\frac{u \psi'}{Ar_0\eta} \Delta d+\frac{u \psi''}{(Ar_0)^2 \eta}+2\langle \nabla \Sh,\nabla \eta \rangle\\
&\le \eta(\langle \nabla \Sh, \nabla f \rangle+\Sh-2|S_{ij}|^2)+\frac{u \psi'}{Ar_0\eta} \Delta d+\frac{u \psi''}{(Ar_0)^2 \eta}-2S \frac{|\nabla \eta|^2}{\eta}\\
&= \eta(\langle \nabla \Sh, \nabla f \rangle+\Sh-2|S_{ij}|^2)+\frac{u \psi'}{Ar_0\eta} \Delta d+\frac{u \psi''}{(Ar_0)^2 \eta}-2u \frac{|\psi'|^2}{\eta^2(Ar_0)^2}\\
&\le \eta(\langle \nabla \Sh, \nabla f \rangle+\Sh-2|S_{ij}|^2)+\frac{u \psi'}{Ar_0\eta}(\langle \nabla f,\nabla d \rangle +2nr_0^{-1}+|\nabla f(p)|)+\frac{u \psi''}{(Ar_0)^2 \eta}-2u \frac{|\psi'|^2}{\eta^2(Ar_0)^2}.
\end{align*}
Remark that at the point $p_{min}$, the first term on the right hand side is $-\eta \Sh \frac{\langle \nabla \eta,\nabla f \rangle}{\eta}=-\frac{u \psi'}{Ar_0\eta}\langle \nabla f,\nabla d \rangle$, so we have a cancellation. Rewrite the inequality as
\begin{equation}
u-2\eta|S_{ij}|^2+\frac{2n u \psi'}{Ar_0^2\eta}+\frac{u \psi'}{Ar_0\eta}|\nabla f(p)|+\frac{u \psi''}{(Ar_0)^2 \eta}-2u \frac{|\psi'|^2}{\eta^2(Ar_0)^2} \ge 0.
\label{3.9}
\end{equation}
Using Cauchy-Schwarz $-2 \eta |S_{ij}|^2 \le -2 \eta \Sh^2/n=-2u^2/(n \eta)$, thus multiplying (\ref{3.9}) by $\eta$ we get
\begin{equation*}
\eta u-\frac{2u^2}{n}+\frac{2n u \psi'}{Ar_0^2}+\frac{u \psi'}{Ar_0}|\nabla f(p)|+\frac{u \psi''}{(Ar_0)^2}-2u \frac{|\psi'|^2}{\eta(Ar_0)^2} \ge 0,
\end{equation*}
and exploiting the bound on the cut-off function
\begin{equation*}
-\frac{2u^2}{n}+\frac{4n u}{Ar_0^2}+\frac{2|u|}{Ar_0}|\nabla f(p)|+6 \frac{|u|}{(Ar_0)^2} \ge 0,
\end{equation*}
so we must have
\begin{equation*}
|u| \le \frac{C(|\nabla f(p)|,n)}{Ar_0^2},
\end{equation*}
for any constant $A$, so $\Sh(p)=u(p) \ge u(p_{min})\ge -\frac{C(|\nabla f(p)|,n)}{Ar_0^2}$ and since this estimate is uniform in $A$, passing to the limit $A \rightarrow \infty$ we get $\Sh(p) \ge 0$. Since $p$ can be chosen arbitrarily this is valid on the whole $M$, concluding the proof of the first assertion.
\bigskip

Recall that (\ref{solitonstrange}) states $\Sh+|\nabla f|^2=f$, so the previous step yields $|\nabla f|^2 \le f$. Equivalently, we have that $|\nabla(\sqrt{f})| \le \frac{1}{2}$, therefore $\sqrt{f}$ grows at most linearly being a Lipschitz function, and so does $|\nabla f| \le \sqrt{f}$; clearly, this implies that $\nabla f$ is a complete vector field.
\bigskip

Finally, suppose there exists a $p \in M$ such that $\Sh(p)=0$. From the first part of the theorem, such a point is a minimum for the $u$ defined above for every $A$. By the strong maximum principle, which we can apply since the completeness of the metric implies $u$ has compact support, we have $u \equiv 0$ for every $A$, so $\Sh \equiv 0$ on $B(p,Ar_0)$ for every $A$, that is $\Sh \equiv 0$ on $M$. Restoring the time dependence, from the equation (\ref{ellipticsoliton}) we have $\mathcal{S}=0$ and $\tau_g \phi=0$, so $g(t)=g_0$ and $\phi(t)=\phi_0$. Rewriting the soliton solutions equation we get 
\begin{equation*}
\Hess(f(t))=-\sigma(t)g_0.
\end{equation*}
Being a shrinking soliton solution, we know that $\sigma(t)<0$ for any $t$, thus $\Sigma(t)=\int_0^t{\sigma(s) ds}$ is monotone decresing as is $c(t) \coloneqq \exp(2 \Sigma(t))$. Defining $\Psi_t$ as the diffeomorphism induced by $\nabla f (t)/c(t)$ (recall it is a complete vector field), we have $\hat{g}=(\Psi_t^*)^{-1}g_0=id$, thus $\hat{g}$ is flat. From the Killing-Hopf Theorem, its universal cover is the Euclidean space $(\R^n,g_{can})$, with covering map $\pi$. Pulling back the above equation on $\R^n$ via $\pi$, we obtain that the function $\pi^*f$ is strictly convex. Therefore $\pi$ must be trivial, otherwise $\pi^* f$ would be periodic, and thus a global isometry. Since $\Sh\equiv \Sc\equiv 0$ and $\alpha \neq 0$, we deduce that $\phi_0=y_0$ is a constant map.
\end{proof}

\section{Reduced Length and Volume Based at Singular Time}
\label{srl}
\subsection{Reduced Distance in Harmonic Ricci Flow}
The concepts of reduced length and volume in Ricci flow were introduced by Perelman in \cite{per}; he proved the monotonicity of this volume, a property which played a key role in his proof of one of the local $\kappa-$non-collapsing theorems essential for his resolution of the Poincar\'e and Thurston Geometrization Conjectures. From the singularities analysis point of view, this concept has the limitation of being based at a regular time: a suitable blow up limit of a Ricci flow has constant reduced volume, thus from the theory of Perelman it is a gradient shrinking soliton solution, but if the blow up procedure is done at a regular space-time point, the curvature boundedness implies directly that this limit is a Gaussian soliton. Enders (in \cite{end}) and Naber (in \cite{nab}) therefore independently introduced a concept of reduced length based at a singular time for Ricci flow, which allows to rescale around singular points where the blow up limits can be non-flat.

\noindent Analogous concepts for harmonic Ricci flow were introduced by M\"uller in \cite{mul1}, where the author proved a monotonicity result and a local $\kappa-$non-collapsing theorem; we refer the reader to it for further background reading. In this section we develop concepts of reduced length and volume based at singular time for the harmonic Ricci flow. Throughout the rest of the paper the harmonic Ricci flows we are considering are \textit{complete} unless otherwise stated.
\begin{proof}[Proof of Theorem \ref{redlensingtime}]
Fix a sequence of time $t_i \nearrow T$ and a point $p\in M$. To simplify notation let $l_i \coloneqq l_{p,t_i}$ and $L_i \coloneqq L_{p,t_i}$. Fix an arbitrary compact subset $K=K_1 \times [a,b] \subset M \times (0,T)$. The plan of the proof is to apply Arzel\`a-Ascoli's theorem to the sequence, once we have proved the necessary bounds. By definition of $l_i$ it suffices to uniformly bound $L_i$ on $K$. Let $\eta \colon [0,1] \rightarrow M$ be a $g(0)-$geodesic with $\eta(0)=q$, $\eta(1)=p$. Fix $k \in (b,T)$ and consider
\begin{equation}
\gamma(t) \coloneqq 
\begin{cases*}
\eta\big(\frac{t-\bar{t}}{k-\bar{t}}\big) \ \ t\in [\bar{t},k], \\
p \ \ \ \ \ \ \ \ \ t \in (k,t_i].
\end{cases*}
\end{equation}
Because $|\eta'(s)|^2_{g(0)}$ is constant (depending on $d_0(p,q)$, and hence only on $K$), there exists a constant $D=D(C,K)$ such that $|\eta'(s)|^2_{g(t)}\le D$ because of the uniform equivalence of the metrics along the harmonic Ricci flow on $[0,k]$. An application of Theorem \ref{typerhrf}, using the Type A assumption, implies the existence of a constant $C=C(\underline{\alpha},\overline{\alpha},n,N,C,r)$ such that $|\Sh|(x,t) \le \frac{C}{(T-t)^r}$, so:
\begin{align}
|L_i(q,\bar{t})| &\le \Bigg| \int_{\bar{t}}^{t_i}{\sqrt{t_i-t}\big(|\dot{\gamma}(t)|^2_{g(t)}+\Sh_{g(t)}(\gamma(t)) \big) dt}  \Bigg| \notag\\ 
&\le \int_{\bar{t}}^{k}{\frac{\sqrt{t_i-t}}{(k-\bar{t})^2} \bigg| \eta' \bigg( \frac{t-\bar{t}}{k-\bar{t}}\bigg) \bigg|^2_{g(t)} dt}+ C \int_{\bar{t}}^{t_i}{\frac{\sqrt{t_i-t}}{(T-t)^r} dt} \\
&\le \frac{D \sqrt{T}}{k-b}+\frac{2C}{3-2r}T^{\frac{3}{2}-r} \eqqcolon E(\underline{\alpha},\overline{\alpha},n,N,C,r,K) \notag
\end{align}
thus the uniform bound is proved. 
\bigskip

Now we would like a uniform bound on the derivatives of the $L_i$. Let $\gamma_i$ be an $\mathcal{L}-$minimizing $\mathcal{L}-$geodesic from $(q,\bar{t})$ to $(p,t_i)$, where $(q,\bar{t}) \in K$. 

We claim there exists a constant $G=G(\underline{\alpha},\overline{\alpha},n,N,C,r,K)>0$ independent of $i$ such that for all $t \in [\bar{t},k]$, $|\sqrt{t_i-t} \dot{\gamma_i}(t)|^2_{g(t)} \le G$. Denote by $V_i(t) \coloneqq \sqrt{t_i-t} \dot{\gamma_i}(t)$. Using the $\mathcal{L}-$geodesic equation (cfr. formula (4.3) in \cite{mul1})
\begin{equation}
\nabla_{V_i(t)} V_i(t) -2 \sqrt{t_i-t} \mathcal{S}_{g(t)}(V_i(t),\cdot)^\# -\frac{1}{2}(t_i-t) \nabla \Sh_{g(t)}=0,
\end{equation}
we obtain
\begin{align}
\frac{d}{dt} |V_i(t)|^2_{g(t)}&=-2\mathcal{S}(V_i(t),V_i(t))+2 \langle \nabla_{\dot{\gamma_i}(t)} V_i(t),V_i(t) \rangle_{g(t)} \notag\\
&= -2\mathcal{S}(V_i(t),V_i(t))+\frac{2}{\sqrt{t_i-t}} \langle \nabla_{V_i(t)} V_i(t),V_i(t) \rangle_{g(t)} \notag\\
&= -2\mathcal{S}(V_i(t),V_i(t))+\frac{2}{\sqrt{t_i-t}} \langle 2 \sqrt{t_i-t} \mathcal{S}_{g(t)}(V_i(t),\cdot)^\#+\frac{1}{2}(t_i-t) \nabla \Sh_{g(t)},V_i(t) \rangle_{g(t)} \label{boundgradredlen}\\
&= 2\mathcal{S}(V_i(t),V_i(t))+\sqrt{t_i-t} \langle \nabla \Sh_{g(t)},V_i(t) \rangle_{g(t)} \notag \\
&\le \frac{C_1}{(T-t)^r}|V_i(t)|^2_{g(t)}+\frac{C_2}{(T-t)^{\frac{3}{2} r-\frac{1}{2}}}|V_i(t)|_{g(t)}. \notag
\end{align}
Here the constants depend upon the Type A constant $C$ as in Theorem \ref{typerhrf} but are independent of $i$. Having a uniform bound on $V_i(t)$ for $t$ in a compact set of time is necessary to exploit this inequality. We have
\begin{equation}
\int_{\bar{t}}^{t_i}{\frac{1}{\sqrt{t_i-t}} |V_i(t)|^2_{g(t)} dt}=\mathcal{L}(\gamma_i)-\int_{\bar{t}}^{t_i}{\sqrt{t_i-t}\Sh_{g(t)}(\gamma_i(t)) dt} \le \mathcal{L}(\gamma_i)+\frac{2C_1}{3-2r}T^{\frac{3}{2}-r}
\end{equation}
by the Type A assumption, thus the integral mean value theorem gives the existence of $\hat{t_i} \in [\bar{t},k]$ such that 
\begin{equation} 
\frac{1}{\sqrt{t_i-\hat{t_i}}} |V_i(\hat{t_i})|^2_{g(\hat{t_i})}=\frac{1}{k-\bar{t}}\int_{\bar{t}}^{k}{\frac{1}{\sqrt{t_i-t}} |V_i(t)|^2_{g(t)} dt} \le \frac{1}{k-\bar{t}} \bigg(\mathcal{L}(\gamma_i)+\frac{2C_1}{3-2r}T^{\frac{3}{2}-r}\bigg)
\end{equation}
for $i$ large, and since $\sqrt{t_i-\hat{t_i}} \le \sqrt{T}$, we get $|V_i(\hat{t_i})|^2_{g(\hat{t_i})} \le F$ for some constant $F=F(\underline{\alpha},\overline{\alpha},n,N,C,r,K)$. Without loss of generality we can assume that $|V_i(t)|^2_{g(t)} \ge 1$, so the inequality (\ref{boundgradredlen}) becomes for $t \in [a,k]$
\begin{equation}
\frac{d}{dt}|V_i(t)|^2_{g(t)} \le \bigg( \frac{C_1}{(T-k)^r}+\frac{C_2}{(T-k)^{\frac{3}{2} r-\frac{1}{2}}} \bigg)|V_i(t)|^2_{g(t)}=C_3 |V_i(t)|^2_{g(t)}.
\end{equation}
The constant $C_3$ depends on $\underline{\alpha},\overline{\alpha},n,N,C,r,K$. Integration  implies that for all $t \in [a,b] \subset [a,k]$
\begin{equation}
|V_i(t)|^2_{g(t)} \le F e^{C_3(t-\hat{t_i})} \le F e^{C_3T} \eqqcolon G=G(\underline{\alpha},\overline{\alpha},n,N,C,r,K).
\end{equation}
In order to get uniform gradient bounds for $L_i$, we use the first variation formula of the functional $\mathcal{L}$ to obtain $\nabla L_i(q,\bar{t})=-2 \sqrt{t_i-\bar{t}} \dot{\gamma_i}(\bar{t})$, thus $|\nabla L_i(q,\bar{t})| \le \sqrt{2G}$.
Regarding the time derivative bounds we proceed as follows 
\begin{align*}
\frac{\partial}{\partial \bar{t}}L_i(q,\bar{t})&=\frac{d}{d\bar{t}}L_i(q,\bar{t})-\langle \nabla L_i(q,\bar{t}),\dot{\gamma_i}(\bar{t}) \rangle_{g(\bar{t})}\\
&=- \sqrt{t_i-\bar{t}}( |\dot{\gamma_i}(\bar{t})|^2_{g(\bar{t})}+\Sh_{g(\bar{t})}(\gamma_i(\bar{t})) )+2\sqrt{t_i-\bar{t}} |\dot{\gamma_i}(\bar{t})|^2_{g(\bar{t})}\\
&=\frac{1}{\sqrt{t_i-\bar{t}}} |\sqrt{t_i-\bar{t}} \dot{\gamma_i}(\bar{t})|^2_{g(\bar{t})}-\sqrt{t_i-\bar{t}} \Sh_{g(\bar{t})}(\gamma_i(\bar{t}))
\end{align*}
therefore using the Type A bound and what we obtained before, we have for any $(q,\bar{t}) \in K$
\begin{equation}
\frac{\partial}{\partial \bar{t}}L_i(q,\bar{t}) \le \frac{G}{\sqrt{k-b}}+\frac{C}{(T-b)^{r-\frac{1}{2}}} \eqqcolon H=H(\underline{\alpha},\overline{\alpha},n,N,C,r,K).
\end{equation}
Thus for any compact set $K$, the $C^{0,1}$ norm of $l_i$ are uniformly bounded (in terms of $\underline{\alpha},\overline{\alpha},n,N,C,r,K$) and we can extract a limit by Arzel\`a-Ascoli's Theorem.
It remains to show that the differential inequalities for $l_i$ pass to the limit. This follows analogously to the Ricci flow case in \cite{end} and we therefore leave this part of the proof to the reader.
\end{proof}
\begin{remark}
We stress that the uniform bound of the $C^1_{loc}$-norm of the $l_i$, which brings another bound on the $C^{0,1}_{loc}$-norm of the $l_{p,T}$, depends on $\underline{\alpha},\overline{\alpha},n,N$ and on the Type A constants $C,r$.
\label{constantdependence}
\end{remark}

\subsection{Reduced Volume Based at Singular Time}
\label{volume}

\begin{definition}
Let $(M,g(t),\phi(t))$ be a Type A harmonic Ricci flow on $(0,T)$. Fix a point $p \in M$, a sequence $t_i \nearrow T$ and any $l_{p,T}$ and $v_{p,T}$ given by the Theorem \ref{redlensingtime}. We define a \textit{reduced volume based at the singular time} $(p,T)$ to be the function 
\begin{equation}
V_{p,T}(\bar{t}) \coloneqq \int_M{v_{p,T}(q,\bar{t}) \ dvol_{g(\bar{t})}(q)}= \int_M{(4\pi(T-\bar{t}))^{-\frac{n}{2}}e^{-l_{p,T}(q,\bar{t})} \ dvol_{g(\bar{t})}(q)}.
\end{equation}
\end{definition}
From the work in \cite{mul} and Fatou's Lemma, we know that $V_{p,T}(\bar{t}) \le 1$. The proof of the next result reads the same as in \cite{end}, hence we skip it here.
\begin{corollary}
We have $\frac{d}{d\bar{t}} V_{p,T}(\bar{t}) \ge 0$ and $\lim_{\bar{t} \nearrow T} V_{p,T}(\bar{t}) \le 1$.
\end{corollary}
We still need an analogous result in the case the reduced volume is constant in an interval. In order to derive it, we notice that if $V(t_1)=V(t_2)$ then $\square^*v_{p,T}=0$ in $(t_1,t_2)$, so parabolic regularity implies $l_{p,T}$ is smooth. For the convenience of the reader we recall this equation 
\begin{equation}
-\frac{\partial l_{p,T}(q,\bar{t})}{\partial \bar{t}}-\Delta l_{p,T}(q,\bar{t})+|\nabla l_{p,T}(q,\bar{t})|^2-\Sh_{g(\bar{t})}+ \frac{n}{2(T-\bar{t})} = 0.
\end{equation}
Combining with the last equality in Theorem \ref{redlensingtime}, we deduce
\begin{equation}
w_{p,T} \coloneqq \big( (T-\bar{t})(2\Delta l_{p,T}-|\nabla l_{p,T})|^2 +\Sh) + l_{p,T}-n \big)v_{p,T} \equiv 0.
\end{equation}
\begin{theorem}
Suppose $v(q,t)=(4 \pi(T-t))^{-\frac{n}{2}}e^{-l(q,t)}$ solves the adjoint heat equation $\square^* v=0$ under the harmonic Ricci flow. Then defining $w \coloneqq \big( (T-t)(2\Delta l-|\nabla l|^2 +\Sh) + l-n \big)v$ as above, we have
\begin{equation}
\square^* w= -2(T-t) \bigg[ \bigg| S_{ij}+\nabla_i \nabla_j l-\frac{g_{ij}}{2(T-t)} \bigg|^2+\alpha|\tau_g\phi-\nabla_i\phi\nabla_i l|^2-\dot{\alpha}|\nabla \phi|^2 \bigg] v.
\end{equation}
\begin{proof}
Recall that
\begin{equation}
\frac{d \Delta}{dt}=2S_{ij}\nabla_i \nabla_j -2 \alpha \tau_g \phi \nabla_i \phi \nabla_i.
\end{equation}
By definition of $v$, we know that $v^{-1} \nabla v =\nabla l$. Hence we get
\begin{align*}
v^{-1} \square^*w=& \ v^{-1} (-\partial_t-\Delta+\Sh) \Big( \big( (T-t)(2\Delta l-|\nabla l|^2+\Sh) + l-n\big)v \Big)\\
=&-(\partial_t+\Delta) \big( (T-t)(2\Delta l-|\nabla l|^2+\Sh) + l\big)-2\langle \nabla \big( (T-t)(2\Delta l-|\nabla l|^2+\Sh) + l\big),v^{-1} \nabla v \rangle\\
=& \ 2\Delta l-|\nabla l|^2+\Sh-(T-t)(\partial_t+\Delta)(2\Delta l-|\nabla l|^2+\Sh)-(\partial_t+\Delta)l\\
&+2(T-t)\langle \nabla (2\Delta l-|\nabla l|^2+\Sh),\nabla l \rangle+2 |\nabla l|^2.
\end{align*}
Let us analyze more carefully the term $(\partial_t+\Delta)(2\Delta l-|\nabla l|^2+\Sh)$ (for convenience in the time derivative calculation, note that $|\nabla l|^2=|dl|^2$):
\begin{align*}
(\partial_t+\Delta)(2\Delta l-|\nabla l|^2+\Sh)=& \ 2\partial_t(\Delta)l +2\Delta(\partial_t+\Delta)l -(\partial_t+\Delta)|dl|^2 +(\partial_t+\Delta)\Sh\\
=& \ 4S_{ij}\nabla_i \nabla_j l -4 \alpha \tau_g \phi \nabla_i \phi \nabla_i l+2 \Delta(|\nabla l|^2-\Sh)-2\mathcal{S}(dl,dl)-2 \langle \nabla(\partial_t l),\nabla l \rangle \\
&- \Delta(|\nabla l|^2)+\Delta \Sh+\Delta \Sh+2|S_{ij}|^2+2\alpha|\tau_g \phi|^2-\dot{\alpha}|\nabla \phi|^2\\
=& \ 4S_{ij}\nabla_i \nabla_j l -4 \alpha \tau_g \phi \nabla_i \phi \nabla_i l+ \Delta(|\nabla l|^2)-2\mathcal{S}(dl,dl)\\
&-2 \langle \nabla(-\Delta l+|\nabla l|^2-\Sh),\nabla l \rangle+2|S_{ij}|^2+2\alpha|\tau_g \phi|^2-\dot{\alpha}|\nabla \phi|^2,
\end{align*}
where we used $\square^* v=0$ in the last equality (written in term of $l$). Substituting this we obtain
\begin{align*}
v^{-1} \square^*w=& \ 2\Delta l-|\nabla l|^2+\Sh-(T-t)\Big[4S_{ij}\nabla_i \nabla_j l -4 \alpha \tau_g \phi \nabla_i \phi \nabla_i l+ \Delta(|\nabla l|^2)-2\mathcal{S}(dl,dl)\\
&-2 \langle \nabla(-\Delta l+|\nabla l|^2-\Sh),\nabla l \rangle+2|S_{ij}|^2+2\alpha|\tau_g \phi|^2-\dot{\alpha}|\nabla \phi|^2 \Big]-\frac{n}{2(T-t)}+\Sh-|\nabla l|^2\\
&+2(T-t)\langle \nabla (2\Delta l-|\nabla l|^2+\Sh),\nabla l \rangle+2 |\nabla l|^2\\
=&-\frac{n}{2(T-t)}+\Big[ 2\Delta l- |\nabla l|^2+\Sh+\Sh-|\nabla l|^2\Big]-(T-t)\Big[4S_{ij}\nabla_i \nabla_j l -4 \alpha \tau_g \phi \nabla_i \phi \nabla_i l\\
&+ \underline{\Delta(|\nabla l|^2)}-\underline{2\mathcal{S}(dl,dl)}-\underline{2 \langle \nabla(\Delta l),\nabla l \rangle}+2|S_{ij}|^2+2\alpha|\tau_g \phi|^2-\dot{\alpha}|\nabla \phi|^2 \Big]\\
=&-\frac{n}{2(T-t)}+2(\Delta l+\Sh)-(T-t)\Big[4S_{ij}\nabla_i \nabla_j l -4 \alpha \tau_g \phi \nabla_i \phi \nabla_i l+2|S_{ij}|^2+2\alpha|\tau_g \phi|^2\\
&-\dot{\alpha}|\nabla \phi|^2+I \Big],
\end{align*}
where $I$ is the sum of the underlined terms, i.e. 
\begin{align*}
I&=\Delta(|\nabla l|^2)-2\mathcal{S}(dl,dl)-2 \langle \nabla(\Delta l),\nabla l \rangle= \nabla_i \nabla_i(\nabla_j l \nabla_j l)-2S_{ij} \nabla_i l \nabla_j l-2\nabla_i \nabla_j \nabla_j l \nabla_i l\\
&=2\nabla_i \nabla_i \nabla_j l \nabla_j l+2\nabla_i \nabla_j l \nabla_i \nabla_j l-2R_{ij} \nabla_i l \nabla_j l+2 \alpha \nabla_i \phi \nabla_i l \nabla_j \phi \nabla_j l -2\nabla_i \nabla_i \nabla_j l \nabla_j l -2 R_{ijjk} \nabla_k l \nabla_i l\\
&= 2|\Hess(l)|^2+2 \alpha \nabla_i \phi \nabla_i l \nabla_j \phi \nabla_j l.
\end{align*}
Now we can plug this expression in the equation above, obtaining
\begin{align*}
v^{-1} \square^*w=&-\frac{n}{2(T-t)}+2(\Delta l+\Sh)-(T-t)\Big[4S_{ij}\nabla_i \nabla_j l -4 \alpha \tau_g \phi \nabla_i \phi \nabla_i l+2|S_{ij}|^2+2\alpha|\tau_g \phi|^2-\dot{\alpha}|\nabla \phi|^2\\
&+2|\Hess(l)|^2+2 \alpha \nabla_i \phi \nabla_i l \nabla_j \phi \nabla_j l \Big]\\
=&-2(T-t)\bigg[ \bigg( |S_{ij}|^2+|\nabla_i l \nabla_j l|^2+ \frac{n}{4(T-t)^2}+2S_{ij}\nabla_i \nabla_j l-\frac{\Delta l}{T-t}-\frac{\Sh}{T-t} \bigg)\\
&+\alpha \big(\nabla_i \phi \nabla_i l \nabla_j \phi \nabla_j l-2\tau_g \phi \nabla_i \phi \nabla_i l+|\tau_g \phi|^2 \big)-\dot{\alpha}|\nabla \phi|^2 \bigg]\\
=&-2(T-t)\bigg[ \bigg| S_{ij}+\nabla_i \nabla_j l -\frac{g_{ij}}{2(T-t)} \bigg|^2+\alpha | \tau_g \phi-\nabla_i \phi \nabla_i l |^2-\dot{\alpha}|\nabla \phi|^2 \bigg]. \qedhere
\end{align*}
\end{proof}
\label{conjheatlocvol}
\end{theorem}
In the next theorem, we follow the argument in \cite{cho} for the similar Ricci flow case.
\begin{theorem}
Suppose we are given a complete gradient shrinking soliton solution in canonical form, i.e. a solution of (\ref{solitoncanonical}) with $a=1$. Then for any point $p \in M$, its reduced length based at the singular space-time point $(p,T)$ equals the soliton potential plus the normalization constant. In particular, it is independent of the point $p$.
\begin{proof}
We know that $g(t)=(T-t) \psi^*_t g(0)$, $\phi(t)=\psi^*_t \phi(0)$ and $f(t)=\psi_t^* f(0)$, where $\psi_t$ is the diffeomorphism generated by $\nabla f(0)/(T-t)$. Therefore, defining $\tau=T-t$, we know that $\frac{\partial f}{\partial \tau}=-|\nabla f|^2$. Without loss of generality, we can suppose the soliton solution to be normalized (note that this requires adding a constant independent of time), so $\Sh+|\nabla f|^2-\frac{1}{\tau}f=0$. Given any curve $\gamma \colon [0,\bar{\tau}] \longrightarrow M$, such that $\gamma(0)=p$, $\gamma(\bar{\tau})=q$, we have
\begin{align*}
\frac{d}{d\tau}(\sqrt{\tau}f(\gamma(\tau),\tau))&=\sqrt{\tau} \bigg( \frac{f}{2 \tau}+\frac{\partial f}{\partial \tau}+\nabla f \cdot \dot{\gamma} \bigg)=\sqrt{\tau} \Big( \frac{f}{2 \tau}+\frac{\partial f}{\partial \tau}+\frac{1}{2}|\nabla f|^2+ \frac{1}{2} |\dot{\gamma}|^2-\frac{1}{2}|\dot{\gamma}-\nabla f|^2 \Big)\\
&=\frac{1}{2}\sqrt{\tau} \Big(\Sh+|\dot{\gamma}|^2 -|\dot{\gamma}-\nabla f|^2 \Big).
\end{align*}
Hence integrating from $0$ to $\bar{\tau}$
\begin{equation}
f(\gamma(\bar{\tau}),\bar{\tau}))= \frac{1}{2 \sqrt{\bar{\tau}}} \mathcal{L}(\gamma)-\frac{1}{2 \sqrt{\bar{\tau}}} \int_0^{\bar{\tau}}{\sqrt{\tau}|\dot{\gamma}(\tau)-\nabla f(\gamma(\tau),\tau)|^2_{g(\tau)} d\tau}.
\label{potentiallength}
\end{equation}
Now choosing $\gamma$ such that $\dot{\gamma}=\nabla f(\tau)$, we get $f(q,\bar{\tau})\ge l_{p,0}(q,\bar{\tau})$. For the opposite inequality, it is sufficient to apply (\ref{potentiallength}) to a minimal $\mathcal{L}$-geodesic since the integrand is non-negative.
\end{proof}
\end{theorem}
Combining the two theorems above we get the following result.
\begin{corollary}
Any complete normalized gradient shrinking soliton solution in canonical form has constant reduced volume based at singular time.
\label{shrinkerconstant}
\end{corollary}
We now want to enquire for a converse of this statement. Suppose the coupling function is non increasing, and suppose that the reduced volume based at the singular time $T$ is constant. Then we have
\begin{equation}
\begin{cases}
\Ric-\alpha \nabla \phi \otimes \nabla \phi+\Hess(l_{p,T})-\frac{g}{2(T-t)}=0,\\
\tau_g \phi-\langle \nabla \phi, \nabla l_{p,T} \rangle=0,\\
\dot{\alpha}|\nabla \phi|^2=0.
\end{cases}
\end{equation}
The first two equations are the equations characterizing a normalized gradient shrinking soliton solution in canonical form with potential function $l_{p,T}$. 
\bigskip

Suppose now that $\alpha$ is not constant, then there exists $\bar{t} \in (0,T)$ such that $\dot{\alpha}(\bar{t})<0$. Therefore the third equation yields $\nabla \phi=0$, i.e. $\phi(\bar{t}) \equiv y_0$ where $y_0 \in N$ (here we are assuming connectedness of $M$). Writing the harmonic Ricci flow at $\bar{t}$ we get
\begin{equation}
\begin{cases}
\partial_t g \big|_{t=\bar{t}}= -2\Ric_{g(\bar{t})}=2\Hess(l_{p,T})-\frac{g}{(T-t)},\\
\partial_t \phi \big|_{t=\bar{t}}=\tau_g(y_0)=0.
\end{cases}
\end{equation}
Thus from the time $\bar{t}$ on, the solution coincides with a gradient shrinking Ricci soliton solution and a constant map. Since the flow is smooth on the whole interval $[0,\bar{t}]$, we also get that $\phi_0$ is isotopic to a constant map. This strongly restricts the class of initial data $(M,g_0,N,\phi_0)$ that gives rise to a harmonic gradient shrinking Ricci soliton solution \textit{in canonical form}. For example, if we have $M=N=S^n$, with $\phi_0=id$, and non-constant coupling function, then for \textit{any metric} on the domain and target, the solution of the harmonic Ricci flow will not be a harmonic gradient shrinking Ricci soliton solution in canonical form! Compare this with the discussion in Section \ref{sss}.

Reassuringly, if we blow up a Type I harmonic Ricci flow, we get a harmonic gradient shrinking Ricci soliton solution in canonical form, but the limit of the coupling functions  $\alpha(T-\frac{t}{\lambda})$ will be a constant because $\lambda$ tends to infinity, and we cannot apply the argument above.
\begin{remark}
We point out that if we blow up a \textit{complete} harmonic Ricci flow at a regular time, then we get the Gaussian soliton in canonical form and a constant map from $\R^n \rightarrow N$.
Indeed suppose we have constant reduced volume based at a regular time $t_0 < T$, then we can assume without loss of generality that $\alpha$ is constant: otherwise we get as above that $\phi \equiv y_0$ is a constant map, and $(g(t),l_{p,t_0})$ is a gradient shrinking Ricci soliton solution, which verifies the Ricci soliton equations with respect to a regular time, thus from Ricci flow theory it must be the Gaussian soliton in canonical form. 

In the case $\alpha$ is constant, we get the more general definition of a gradient shrinking harmonic Ricci soliton solution in canonical form, i.e. $(g(t),\phi(t),l_{p,t_0})$ solves the system:
\begin{equation}
\begin{cases}
\Ric-\alpha \nabla \phi \otimes \nabla \phi+\Hess(l_{p,t_0})-\frac{g}{2(t_0-t)}=0,\\
\tau_g \phi-\langle \nabla \phi, \nabla l_{p,t_0} \rangle=0.
\end{cases}
\end{equation}
We claim that this can only be the case if $(g(t),y_0,l_{p,t_0})$ is the Gaussian soliton in canonical form. Recall that by the long time existence result in \cite{mul}, we have bounded curvature up to $t_0$, so we can use the maximum principle. 
By Theorem \ref{rigidity}, we have $\Sc \ge \Sh \ge 0$. Moreover, by the same theorem we know that there exists a family $\psi_t$ of diffeomorphisms induced by $\nabla l_{p,t_0}(t)$, so $\Sh(t,x)=(t_0-t)^{-1}\Sh(0,\psi_t(x))$. Taking the limit $t \rightarrow t_0$, if there exists an $x_0$ such that $\Sh(x_0,0)>0$, then $\Sh(t_0,\psi_{t_0}(x_0))=\infty$, which is contradictory to the assumption that $t_0$ is a regular time. Thus $\Sh \equiv 0$ and we conclude using the rigidity part in Theorem \ref{rigidity}. 
\end{remark}
\section{Pseudolocality Theorem for Harmonic Ricci Flow}
\label{sp}
\label{pseudolocality}
In this section we want to prove a pseudolocality theorem for the harmonic Ricci flow in the same spirit as \cite{guo} where this is done for List's flow (i.e. harmonic Ricci flow with target $N=\R$). The proof is similar to the analoguous Ricci flow case (see \cite{per},\cite{kle}), and we want to stress that we had to built all the machinery in the previous sections to make it rigorous. For this reason we will emphasize the parts of the proof which needed a correction and refer the reader back to \cite{guo} for the remaining ones.

First, recall the following \textit{Perelman-Li-Yau-Harnack type} inequality due to B\u{a}ile\c{s}teanu and Tran.
\begin{theorem}[Theorem 1.1 in \cite{bai}]
Let $(M,g(t),\phi(t))$ be a harmonic Ricci flow on $[0,T]$, with non-increasing coupling function $\alpha(t)$. Suppose $u \coloneqq (4 \pi (T−t))^{-n/2} e^{-f}$ is a fundamental solution of the adjoint heat equation $\square^* u=0$, i.e. it tends to $\delta_p$ as $t \rightarrow 0$, where $\delta_p$ is the Dirac delta at a certain point $p \in M$. Then defining $v \coloneqq ((T-t)(2 \Delta f-|\nabla f|^2+\Sh)+f-n)u$, we have $v \le 0$ in $(0,T]$.
\label{plyhinequality}
\end{theorem}
\noindent Equipped with this result, we can now give a proof of the Pseudolocality Theorem \ref{pseudo}.
\begin{proof}[Proof of Theorem \ref{pseudo}]
First of all, we can assume without loss of generality that $\beta < 1/100n$ and $r_0=1$. Arguing by contradiction, suppose that there exist $\e_k,\delta_k \longrightarrow 0$ such that for each $k$, there exist a complete pointed harmonic Ricci flow $(M_k,g_k,\phi_k,p_k)$, with fixed target manifold $N$, such that $\Sh_{g_k}(0) \ge -1$ on a ball $B_{g_k(0)}(p_k,1)$ and for any $\Omega \subset B_{g_k(0)}(p_k,1)$ we have $\Area_{g_k(0)}(\partial \Omega)^n \ge (1-\delta_k)c_nVol_{g_k(0)}(\Omega)^{n-1}$, and points $(x_k,t_k)$ such that $0<t_k \le \e_k^2$, $d_{g_k(t_k)}(x_k,p_k)\le \e_k$, but
\begin{equation*}
|\Rm|(x_k,t_k)\ge \beta t_k^{-1} + \e_k^{-2}.
\end{equation*}
Moreover, we can choose $\e_k$ small enough such that
\begin{equation}
|\Rm|(x,t)\ge \beta t_k^{-1} + 2\e_k^{-2},
\label{localbound}
\end{equation}
whenever $0<t\le \e_k^2$ and $d_{g_k(t)}(x,p_k)\le \e_k$.
We can use precisely the same point selection argument as in \cite{kle} to get the following.
\begin{lemma}
Given a sequence $A_k \longrightarrow \infty$, for any $k$ there exists a space-time point $(\overline{x}_k,\overline{t}_k) \in M_k \times (0, \e_k]$ with $d_{g_k(\overline{t}_k)}(\overline{x}_k,p_k)\le (1+2 A_k)\e_k$, such that
\begin{equation*}
|\Rm|_{g_k(t)}(x,t)\le 4Q_k \coloneqq 4 |\Rm|(\overline{x}_k,\overline{t}_k),
\end{equation*}
for any $(x,t)$ in the backward parabolic region 
\begin{equation*}
\Omega_k \coloneqq \bigg\{(x,t) \bigg| d_{g(\overline{t}_k)}(x,\overline{x}_k) \le \frac{1}{10}A_k Q_k^{-1/2}, \overline{t}_k-\frac{1}{2}\beta Q_k^{-1} \le t \le \overline{t}_k \bigg\} .
\end{equation*}
\label{pickpoint}
\end{lemma}
\noindent For each $k$, let $u_k \coloneqq (4 \pi (T−t))^{-n/2} e^{-f_k}$ be the fundamental solution of the adjoint heat equation $\square_k^* u_k=0$, which tends to $\delta_{\overline{x}_k}$ as $t \rightarrow 0$. Define as above $v_k \coloneqq ((T-t)(2 \Delta f_k-|\nabla f_k|^2+\Sh_k)+f_k-n)u_k$. Then we know by the Theorem \ref{plyhinequality} that $v_k \le 0$. The proof of the next lemma requires a proper compactness theorem (the proof given in \cite{guo} for List's flow can be corrected using our Compactness Theorem \ref{compact2}).
\begin{lemma}
There exist a constant $b>0$ independent of $k$ and times $\tilde{t}_k \in (t_k-\frac{1}{2}\beta Q_k^{-1}, \overline{t}_k)$ such that 
\begin{equation*}
\int_{B_k}{v_k dV_{g_k(\tilde{t_k})}} \le -b <0,
\end{equation*}
where $B_k=B_{g_k(\tilde{t_k})}\big(\overline{x}_k,\sqrt{\overline{t}_k-\tilde{t}_k}\big)$.
\begin{proof}
This proof is by contradiction. Suppose that for any $\tilde{t}_k \in (t_k-\frac{1}{2}\beta Q_k^{-1}, \overline{t}_k)$ there exists a subsequence (not relabelled) such that 
\begin{equation}
\liminf_{k \rightarrow \infty}{\int_{B_k}{v_k dV_{g_k(\tilde{t_k})}}} \ge 0.
\label{liminf}
\end{equation}
Consider the rescaling $\hat{g}_k(t)=Q_k g_k(Q_k^{-1}t+\overline{t}_k)$, $\hat{\phi}_k(t)= \phi_k(Q_k^{-1}t+\overline{t}_k)$ for $t \in [-Q_k \overline{t}_k,0]$. Note that under this rescaling $\alpha_k(t)$ becomes $\hat{\alpha}_k(t)=\alpha_k(Q_k^{-1}t+\overline{t}_k)$, while the set $\Omega_k$ becomes the parabolic region $B_{\hat{g}_k(0)}\Big(\overline{x}_k,\frac{1}{10}A_k\Big)\times [-\frac{1}{2} \beta,0]$, where we have curvature bounded by $4$. We split the proof in two cases: 

\noindent \textbf{Case $1$}: Suppose first that the injectivity radii of the $\hat{g}_k$ are uniformly bounded away from zero. We can use Theorem \ref{compact2} to extract a subsequence of $(M_k,\hat{g}_k(t),\hat{\phi}_k(t),\overline{x}_k)$ converging to a harmonic Ricci flow $(M_{\infty},g_{\infty}(t),\phi_{\infty}(t),x_{\infty})$ on $[-\frac{1}{2} \beta,0]$. This limit flow is complete, has $|\Rm| \le 4$, and $|\Rm|(x_{\infty},0)=1$. Notice that the limit coupling function $\alpha_{\infty}$ is constant, hence we can improve the pointwise convergence given by the Theorem \ref{compact2} to a $C^{\infty}([-\frac{1}{2} \beta,0])$-convergence because of the peculiar form of $\hat{\alpha}_k(t)$.

The fundamental solutions $\hat{u}_k$ based at $(\overline{x}_k,0)$ of the rescaled flows will converge smoothly to $u_{\infty}$, the fundamental solution to the conjugate heat equation on the limit flow, based at $(x_{\infty},0)$. So the respective $\hat{v}_k$ converge to the respective $v_{\infty} \le 0$ since it is a pointwise limit of non-negative functions. On the other hand, from the inequality (\ref{liminf}), it follows that for any fixed $t_0 \in [-\frac{1}{2} \beta,0]$ we have 
\begin{equation*}
\int_{B_{g_{\infty}(t_0)}\big(x_{\infty},\sqrt{-t_0}\big)}{v_{\infty}(\cdot,t_0) dV_{g_{\infty}(t_0)}} \ge 0.
\end{equation*}
Therefore it must be $v_{\infty}(\cdot,t_0)=0$ on $B_{g_{\infty}(t_0)}\big(x_{\infty},\sqrt{-t_0}\big)$. A strong maximum principle argument yields $v_{\infty} \equiv 0$ on $M_{\infty} \times (t_0,0]$. 
In particular, we obtain that $\square^* w_{\infty} =0$, where $w_{\infty}$ is the function defined by $ w_{\infty} \coloneqq \big( (T-t)(2\Delta f_{\infty}-|\nabla f_{\infty})|^2+\Sh) + f_{\infty}-n \big)v_{\infty}$, so Theorem \ref{conjheatlocvol} allows us conclude that $(M_{\infty},g_{\infty}(t),\phi_{\infty}(t),f_{\infty}(t))$ is a gradient harmonic Ricci shrinking soliton solution. Since it has bounded curvature and is complete, the results in Section \ref{zhangtype} ensure it is the Gaussian soliton in canonical form; this is contradictory to $|\Rm|(x_{\infty},0)=1$.
\bigskip

\noindent \textbf{Case $2$}: In the case the injectivity radii of the rescaled metrics $\hat{g}_k$ at the point $\overline{x}_k$ tend to zero, denote them by $r_k=inj(\overline{x}_k,\hat{g}_k(0))\rightarrow 0$. Rescale further by $\tilde{g}_k(t)=r_k^{-2} \hat{g}_k(r_k^2 t)$, $\tilde{\phi}_k(t)= \hat{\phi}_k(r_k^2 t)$, $\tilde{\alpha}_k(t)=\hat{\alpha}_k(r_k^2 t)$, defined for $t \in [-\frac{1}{2} \beta r_k^{-2},0]$. Clearly, the injectivity radius of $\tilde{g}_k(0)$ at $\overline{x}_k$ is $1$ for every $k$. The region $\Omega_k$ becomes $B_{\tilde{g}_k(0)}\big(\overline{x}_k,\frac{1}{10}A_k r_k^{-1}\big)\times \big[-\frac{1}{2} \beta r_k^{-2},0\big]$. Remark that on these regions, which are larger and larger (both in space and time) with $k$ increasing, we have curvature bounded by $4 r_k^2 < \infty$ uniformly in $k$, so we can use Theorem \ref{compact2} to get a subsequence of $(M_k,\tilde{g}_k(t),\tilde{\phi}_k(t),\overline{x}_k)$ converging in the pointed Cheeger-Gromov sense to a complete pointed harmonic Ricci flow $(M_{\infty},g_{\infty}(t),\phi_{\infty}(t),x_{\infty})$ for $(-\infty,0]$, with constant coupling function (again with smooth convergence). Moreover, the curvature bound gives that $g_{\infty}$ is a flat metric and hence $\phi_{\infty}$ is a constant map since the limit flow is ancient and complete.

Passing to the limit (\ref{liminf}), we get as above that $(M_{\infty},g_{\infty}(t),\phi_{\infty}(t),f_{\infty}(t))$ is a gradient harmonic Ricci shrinking soliton solution. Again, the completeness and the boundedness of the curvature imply by the discussion in Section \ref{zhangtype} that the limit flow has to be the Gaussian soliton in canonical form. In particular, the injective radius $inj_{g_{\infty}}(x_{\infty},0)=+\infty$, which is contradictory.
\end{proof}
\label{gaplemma}
\end{lemma}
\noindent From here on the proof of Theorem \ref{pseudo} reads the same as in \cite{guo}: One can first show for every $k$ a similar uniform integral bound as in Lemma \ref{gaplemma} for the time $\bar{t}_k$ slice and certain compactly supported functions, using the completeness assumption and the Harnack inequality. Then asymptotically confronting this bound with the log-Sobolev inequality, which is valid for domains close to the euclidean space as by assumption, one obtains a contradiction. We refer the reader to \cite{guo} for details.
\end{proof}
\begin{remark}
Remark that the hypothesis $(c)$ appearing in Theorem 1 of \cite{guo} is absorbed here in the assumption that $N$ is closed.
\end{remark}
\noindent We present a slightly modified version of the Pseudolocality Theorem, which is a corollary of Theorem \ref{pseudo}.
\begin{theorem}[HRF Pseudolocality Theorem: Version $2$]
There exist $\e,\delta>0$ depending as above on $\underline{\alpha},\overline{\alpha}$ and $N$ with the following
property. Let $(M,g(t),\phi(t), p)$ be a smooth complete pointed harmonic Ricci flow solution  defined for $t \in [0,(\e r_0)^2]$. Assume further the following conditions:
\begin{itemize}
\item $|\Rm|(0) \le r_0^{-2}$ on $B_{g(0)}(p,r_0)$;
\item $\Vol_{g(0)}(B_{g(0)}(p,r_0)) \le (1-\delta) \omega_n r_0^n$, where $\omega_n$ is the volume of the Euclidean unit ball.
\end{itemize}
Then $|\Rm|(x,t)<(\e r_0)^{-2}$ whenever $0\le t \le (\e r_0)^2$ and $d_{g(0)}(x,p)\le \e r_0$.
\begin{proof}
First of all, we notice that Theorem \ref{pseudo} guarantees the existence of $\e'$ and $\delta$ such that for every flow verifying the hypothesis in that theorem we have $|\Rm|(x,t)\le \beta t^{-1} + (\e r_0)^{-2}$ for $0 \le t \le (\e r_0)^2$ and $d(x,t) \le \e r_0$. Using the continuity and the initial data bound, for $\e''$ smaller than $2$ (say), we have $|\Rm|(x,t)<(\e'' r_0)^{-2}$ for $0 \le t \le t_0(\e'')$ and $d(x,t) \le r_0$. Choosing a possibly smaller $\e$ (depending only on the previous $\e'$ and $\e''$), we get $|\Rm|(x,t)<(\e r_0)^{-2}$ for $0\le t \le (\e r_0)^2$ and $d_{g(t)}(x,p)\le \e r_0$. Finally, the length distortion Lemma \ref{lengthdist} gives us the conclusion after possibly further decreasing $\e$ in dependence of the previous $\e$.
\end{proof}
\label{pseudo2}
\end{theorem}

\section{Proof of the Main Theorem}
\label{smt}
\label{typeI}
In this section we apply the theory developed above to the study of Type I singularities in the harmonic Ricci flow. We follow the structure in \cite{end1}.
\begin{definition}
Given $(M,g(t),\phi(t))$ a harmonic Ricci flow on $[0,T)$, a sequence of space-time points $(p_i,t_i) \in M \times [0,T)$ with $t_i \nearrow T$ is called an \textit{essential blow-up sequence} if there exists a constant $c>0$ such that
\begin{equation*}
|\Rm_{g(t_i)}|_{g(t_i)}(p_i) \ge \frac{c}{T-t_i}.
\end{equation*}
If $(M,g(t),\phi(t))$ is a Type I harmonic Ricci flow, a point $p \in M$ is called a \textit{Type I singular point} if there exists an essential blow-up sequence $(p_i,t_i)$ with $p_i \rightarrow p$. We denote by $\Sigma_I$ the set of Type I singular points.
\label{definitiontypeIsing}
\end{definition}
\begin{proof}[Proof of Theorem \ref{nontrivial}]
Fix a sequence $\lambda_j \longrightarrow +\infty$. Notice that it suffices to show that we can extract a subsequence converging to a limit as claimed in order to conclude that every possible limit flow verifies the same properties. Therefore we would like to use one of the compactness theorems obtained in Section \ref{compactness}. First of all, we note that the Type I assumption is preserved by the rescaling above:
\begin{equation}
|\Rm_{g_j(t)}|_{g_j(t)}(x)= \frac{1}{\lambda_j} \big|\Rm_{g\big(T+\frac{t}{\lambda_j}\big)}\big|_{g\big(T+\frac{t}{\lambda_j}\big)}(x) \le \frac{C}{\lambda_j\big(T-\big(T+\frac{t}{\lambda_j}\big)\big)}=\frac{C}{-t}.
\label{rescaledtypeI}
\end{equation}
This gives uniform curvature bounds only on compact subsets of $(-\infty,0)$. In order to have convergence on the full time domain, we may use the compactness Theorem \ref{compact2} on the time interval $[-n,1/n]$ for every $n \in \mathbb{N}^{+}$ and then use a diagonal argument to obtain a limit flow on $(-\infty,0)$. We need to check the uniform injectivity radii bound at the time, say, $t=-1$. We proceed as follows.

Let $l_{p,T}$ be any fixed reduced length based at the singular space-time point $(p,T)$ for the harmonic Ricci flow $(M,g(t),\phi(t))$, whose existence is guaranteed by Theorem \ref{redlensingtime}. For every space-time point $(q,\bar{t}) \in M \times (-\infty,0)$, it makes sense to consider for large enough $j$
\begin{equation*}
l^j_{p,0}(q,\bar{t}) \coloneqq l_{p,T}\Big(q, T+\frac{\bar{t}}{\lambda_j}\Big),
\end{equation*}
which is, by the scaling properties of the reduced length, a reduced length based at the singular time $t=0$  for the rescaled harmonic Ricci flow. We stress that these functions are locally uniformly Lipschitz continuous, since Remark \ref{constantdependence} guarantees that their Lipschitz norm depend only upon the compact subset chosen, the Type I bound (which is uniform in $j$ by (\ref{rescaledtypeI})) and other quantities that depend uniformly on $j$ . The corresponding reduced volumes verify $V^j_{p,0}(\bar{t})=V_{p,T}\big(T+\frac{\bar{t}}{\lambda_j}\big)$, and are uniformly bounded on compact subsets of $(-\infty,0)$. Thus we can assume, possibly taking a subsequence, that $V^j_{p,0}(\bar{t})$ is pointwise convergent in $(-\infty,0)$. Because of the monotonicity of the reduced volume based at the singular time, this limit is the constant $\lim_{t \nearrow T} V_{p,T}(t) \in (0,1]$, which is continuous, so the convergence is uniform on compact subsets of $(-\infty,0)$. In particular, $V^j_{p,0}(-1)$ is uniformly bounded away from zero, hence we obtain the uniform injectivity radii bound needed because of the $\kappa-$non-collapsing Theorem $6.13$ in \cite{mul}.

Therefore we can use Theorem \ref{compact2} to get a complete limit flow $(M_{\infty},g_{\infty}(t),\phi_{\infty}(t),p_{\infty})$ on $(-\infty,0)$ (after a diagonal argument).
By the discussion above we can assume that the sequence $l^j_{p,0}$ is converging in $C^{0,1}_{loc}(M_{\infty} \times (-\infty,0))$ to a certain function $l^{\infty}_{p_{\infty},0}$ (here we are pulling-back through the diffeomorphisms given by the Cheeger-Gromov convergence for considering every function as a function on $M_{\infty}$). Since its corresponding formal reduced volume is constant, we can conclude as in Section \ref{volume} that $(M_{\infty},g_{\infty}(t),\phi_{\infty}(t),l^{\infty}_{p_{\infty},0})$ is a normalized gradient shrinking harmonic Ricci soliton solution in canonical form. This proves the first statement.
\bigskip

Suppose now that $p \in \Sigma_I$. Arguing by contradiction, we assume that the limit $g_{\infty}(t)$ is flat for all $t<0$ and that the map $\phi_{\infty}$ is constant. In particular, they are independent of time, and we denote them by $\hat{g}$ and $\hat{\phi}$. Fix any $r_0$ smaller than the injectivity radius of $\hat{g}$ at $p_{\infty}$ (which can be shown to be infinite as we did in Section \ref{sss}), so that $B_{\hat{g}}(p_{\infty},r_0)$ is a Euclidean ball. By Cheeger-Gromov convergence, for any $\e$ smaller than one, we have that $B_{g_j(-(\e r_0)^2)}(p,r_0)$ is as close as we want to a Euclidean ball for every $j$ large enough, as well as $\phi_j$ is as close as we want to the constant map $\hat{\phi}$. Now for any $\beta$, we can pick $\e$ and $\delta$ given by Theorem \ref{pseudo}; after possibly reducing $\e$ we can assume that $B_{g_j(-(\e r_0)^2)}(p,r_0)$ verifies the assumptions in Theorem \ref{pseudo}, hence we have
\begin{equation}
|\Rm|_{g_j}(t,x) \le \beta (t+(\e r_0)^2)^{-1}+(\e r_0)^{-2} \ \ \ for \ \ -(\e r_0)^2 \le t<0, \ x \in B_{g_j(-(\e r_0)^2)}(p,\e r_0).
\label{firstboh}
\end{equation}
Using that $p$ is a Type I singular point, we get the existence of a sequence $(p_i,t_i)$, with $p_i \rightarrow p$, $t_i \nearrow T$ and of a constant $c>0$ such that
\begin{equation*}
|\Rm|_{g_j(\lambda_j(t_i-T))}(p_i) \ge \frac{c}{\lambda_j(T-t_i)}.
\end{equation*}
Therefore, for $i$ large enough we can use both the inequalities to get
\begin{equation*}
\frac{c}{\lambda_j(T-t_i)} \le \frac{\beta}{(\lambda_j(T-t_i)+(\e r_0)^2)} +(\e r_0)^{-2},
\end{equation*}
which yields a contradiction for $i$ large enough, since $T-t_i$ is tending to zero.
\end{proof}
\begin{definition}
Define $\Sigma_S \subset \Sigma_I$ to be the set of points $p$ in $M$ for which there exists a constant $c>0$ such that (for $t$ close to $T$)
\begin{equation*}
|\Sh_{g(t)}|(p) \ge \frac{c}{T-t}.
\end{equation*}
Moreover, we say that $p \in M$ is a \textit{singular point} if there does not exist any neighbourhood $U_p \ni p$ on which $|\Rm_{g(t)}|_{g(t)}$ remains bounded as $t$ approaches $T$. The set of such points is denoted  by $\Sigma$.
\end{definition}
\noindent It is clear from the definitions above that $\Sigma_S \subseteq \Sigma_I \subseteq \Sigma$. More interestingly, we have the following result.
\begin{theorem}
Let $(M,g(t),\phi(t))$ be a complete Type I harmonic Ricci flow on $[0,T)$ with finite singular time $T$, with non-increasing coupling function $\alpha(t) \in [\underline{\alpha},\overline{\alpha}]$, where $0<\underline{\alpha}\le \overline{\alpha}<\infty$. Then $\Sigma=\Sigma_S$, so every definition of singular set given above agrees with the others.
\begin{proof}
Firstly, we prove that $\Sigma_I \subseteq \Sigma_S$. Suppose that $p \in M \setminus \Sigma_S$. Then, there exist a sequence of number $c_j \searrow 0$ and $t_j \in [T-c_j,T)$ such that $\Sh(p,t_j) < \frac{c_j}{T-t_j}$. Let $\lambda_j\coloneqq (T-t_j)^{-1}$, and rescale the harmonic Ricci flow as in Theorem \ref{nontrivial}. The same theorem gives us that $(M,g_j(t),\phi_j(t),p)$ subconverges to a normalized complete gradient shrinking harmonic Ricci soliton in canonical form $(M_{\infty},g_{\infty}(t),\phi_{\infty}(t),p_{\infty})$ on $(-\infty,0)$, with
\begin{equation*}
0 \le \Sh_{g_{\infty}(-1)}(p_{\infty})= \lim \lambda_j^{-1}\Sh_{g(t_j)}(p) \le \lim c_j =0.
\end{equation*}
Here we have used the Type I assumption as well as the first part of Theorem \ref{rigidity}. The second conclusion in the latter theorem ensures that $(M_{\infty},g_{\infty}(t),\phi_{\infty}(t))$ is the Gaussian soliton. By Theorem \ref{nontrivial}, we conclude that $p$ must not be a Type I singular point. This allows us to conclude $\Sigma_I=\Sigma_S$.
\bigskip

In order to prove $\Sigma \subseteq \Sigma_I$ we claim the following: Given any $p \in M \setminus \Sigma_I$, there exists a neighbourhood $U_p$ of $p$ on which the curvature remains bounded. Indeed, given such a point $p$, for any $\lambda_j \rightarrow \infty$, the rescaled harmonic Ricci flow (as in Theorem \ref{nontrivial}) subconverges to the Gaussian soliton in canonical form. As we did in the theorem above, for large enough $j \ge j_0$ we can assume that the hypothesis in the Theorem \ref{pseudo} are verified with $r_0=1$. This yields the existence of constants $\e$ and $\delta$ such that, calling $K=\lambda_{j_0}$, we get
\begin{equation*}
|\Rm_{g_{j_0}(t)}|_{g_{j_0}(t)} \le \beta (t+\e^2)^{-1}+\e^{-2} \ \ \ for \ \ -\e^2 \le t<0, \ x \in B_{g_{j_0}(\e^2)}(p,\e),
\end{equation*}
which is equivalent to
\begin{equation*}
|\Rm_{g(t)}|_{g(t)} \le K(\beta (t+\e^2)^{-1}+\e^{-2}) \ \ \ \forall t \in \Big[T-\frac{\e^2}{K},T\Big)
\end{equation*}
on the neighbourhood $U_p = B_{g(T-\frac{\e^2}{K})}\big(p,\e/\sqrt{K}\big)$ of $p$. The curvature bound for times $t < T-\frac{\e^2}{K}$ easily derives from the Type I condition.
\end{proof}
\end{theorem}
\noindent As a corollary we get the following non-oscillation result:
\begin{corollary}
Let $(M,g(t),\phi(t))$ be a complete Type I harmonic Ricci flow on $[0,T)$ with finite singular time $T$, with non-increasing coupling function $\alpha(t) \in [\underline{\alpha},\overline{\alpha}]$, where $0<\underline{\alpha}\le \overline{\alpha}<\infty$. Then for every $p \in \Sigma_I$ there exists a constant $c>0$ such that
\begin{equation}
|\Rm_g|(p,t) \ge \frac{c}{T-t}.
\end{equation}
\label{nonoscillation}
\end{corollary}
\begin{theorem}
Let $(M,g(t),\phi(t))$ be a complete Type I harmonic Ricci flow on $[0,T)$ with finite singular time $T$, with non-increasing coupling function $\alpha(t) \in [\underline{\alpha},\overline{\alpha}]$, where $0<\underline{\alpha}\le \overline{\alpha}<\infty$. Then $\Vol_{g(0)}(\Sigma) < \infty$ implies that $\Vol_{g(t)}(\Sigma) \longrightarrow 0$ as $t \nearrow T$.
\begin{proof}
By the Type I assumption we can apply the maximum principle to the evolution equation of $\Sh$ to get the existence of a constant $\tilde{C}$ such that $\inf_M \Sh(t) \ge - \tilde{C}$ for all $t \in [0,T)$. Define the following sets
\begin{equation*}
\Sigma_{S,k} \coloneqq \bigg\{p \in M \bigg| \Sh(p,t) \ge \frac{1/k}{T-t}, \forall t \in \bigg(T-\frac{1}{k},T\bigg)\bigg\} \subseteq \Sigma_S=\Sigma,
\end{equation*}
for every $k \in \mathbb{N}^{+}$, and $\Sigma_{S,0} \coloneqq \emptyset$. We claim that for every point $x \in \Sigma_{S,k}$ and for all $t \in [0,T)$ we have
\begin{equation*}
\int_0^t{\Sh(x,s) ds} \ge -\tilde{C}T+\log \bigg( \frac{1/k}{T-t} \bigg)^{1/k}.
\end{equation*}
Clearly, this holds for $t \le T-\frac{1}{k}$, since the argument of the logarithm is less than one and by the initial consideration $\int_0^t{\Sh(x,s) ds} \ge -\tilde{C}t \ge -\tilde{C}T$. Instead for $t \in \big(T-\frac{1}{k},T\big)$, the definition of $\Sigma_{S,k}$ yields
\begin{align*}
\int_0^t{\Sh(x,s) ds}&=\int_0^{T-1/k}{\Sh(x,s) ds}+\int_{T-1/k}^t{\Sh(x,s) ds} \ge -\tilde{C}\bigg(T-\frac{1}{k}\bigg)+\int_{T-1/k}^t{\frac{1/k}{T-s} ds} \\
&\ge -\tilde{C}T+\log \bigg( \frac{1/k}{T-t} \bigg)^{1/k}.
\end{align*}
Using the evolution equation for the volume form, as well as the inequality $k \le 2^k$ for all $k \in \mathbb{N}$, we get
\begin{equation*}
\Vol_{g(t)}(\Sigma_{S,k} \setminus \Sigma_{S,k-1})=\int_{\Sigma_{S,k} \setminus \Sigma_{S,k-1}}{e^{-\int_0^t{\Sh(x,s) ds}} dvol_{g(0)}(x)} \le 2e^{\tilde{C}T}(T-t)^{1/k}\Vol_{g(0)}(\Sigma_{S,k} \setminus \Sigma_{S,k-1}).
\end{equation*}
Notice that
\begin{equation*}
\sum_k{\Vol_{g(0)}(\Sigma_{S,k} \setminus \Sigma_{S,k-1})}=\Vol_{g(0)}(\Sigma_{S}) < \infty,
\end{equation*}
so we can use the (discrete) dominated convergence theorem to conclude that 
\begin{align*}
\limsup_{t \rightarrow T}{\Vol_{g(t)}(\Sigma_{S})}&=\limsup_{t \rightarrow T}{\sum_k{\Vol_{g(t)}(\Sigma_{S,k} \setminus \Sigma_{S,k-1})}}\\
&\le 2e^{\tilde{C}T} \limsup_{t \rightarrow T}{\sum_k{ (T-t)^{1/k}\Vol_{g(0)}(\Sigma_{S,k} \setminus \Sigma_{S,k-1})}}=0.
\end{align*}
\end{proof}

\end{theorem}

\end{document}